%% file: ArXivBruchhaeuser.tex
\documentclass[a4paper,10pt]{article}

\usepackage[utf8]{inputenc}
\usepackage{amsfonts}
\usepackage{amssymb,amsmath,amsthm}

\usepackage{authblk}

\usepackage{tikz}
\usetikzlibrary{plotmarks}
\usetikzlibrary{shapes,arrows}
\usetikzlibrary{positioning}
\usetikzlibrary{trees,mindmap}
\usetikzlibrary{fadings,shapes.arrows,shadows}
\usetikzlibrary{decorations.pathreplacing}
\usetikzlibrary{calc}
\usetikzlibrary{patterns}
\usepackage{pgfplots}
\pgfplotsset{compat=1.10}
\usepgfplotslibrary{fillbetween}

\usepackage{array}
\usepackage{hhline}

\usepackage{booktabs}

\usepackage{subfig}

\usepackage{graphicx}

\newtheorem{defi}{Definition}[section]
\newtheorem{theorem}[defi]{Theorem}

\newtheorem{remark}[defi]{Remark}
\newtheorem{corollary}[defi]{Corollary}

\numberwithin{equation}{section}
\numberwithin{table}{section}
\numberwithin{figure}{section}

\usepackage{geometry}
\begin{document}

\title{\Large \textbf{Numerical study of goal-oriented error control for stabilized 
finite element methods}}
\author[M.\ P.\ Bruchh\"auser, M.\ Bause]
{\large  M.\ P.\ Bruchh\"auser\thanks{bruchhaeuser@hsu-hh.de}\,, 
K.\ Schwegler,
M.\ Bause\thanks{bause@hsu-hh.de}\\
{\small Helmut Schmidt University, Faculty of Mechanical Engineering,\\ 
Holstenhofweg 85, 22043 Hamburg, Germany}
}
\date{}
\maketitle

\begin{abstract}
The efficient and reliable approximation of convection-dominated problems 
continues to remain a challenging task. To overcome the difficulties associated 
with the discretization of convection-dominated equations, stabilization 
techniques and a posteriori error control mechanisms with mesh adaptivity were 
developed and studied in the past. Nevertheless, the derivation of robust a 
posteriori error estimates for standard quantities and in computable norms is 
still an unresolved problem and under investigation. Here we combine the 
Dual Weighted Residual (DWR) method for goal-oriented error control with 
stabilized finite element methods. By a duality argument an error representation 
is derived on that an adaptive strategy is built. 
The key ingredient of this work is the application of a higher order 
discretization of the dual problem in order to make a robust error control for 
user-chosen quantities of interest feasible. By numerical experiments in 2D and 
3D we illustrate that this interpretation of the DWR methodology is capable to 
resolve layers and sharp fronts with high accuracy and to further reduce 
spurious oscillations.
\end{abstract}


\section{Introduction}
\label{BruchhaeuserSec:Intro}
From the second half of the last century to nowadays, especially in the pioneering works 
of the 1980's (cf., 
e.g.,~\cite{BruchhaeuserBH81,BruchhaeuserHMM86}), 
strong efforts and great progress were made in the development of accurate and 
efficient approximation schemes for convection-dominated problems. For a review 
of fundamental concepts related to their analysis and 
approximation and a presentation of prominent robust numerical methods we refer 
to the monograph \cite{BruchhaeuserRST08}. Convection-dominated problems arise in many 
branches of technology and science. Applications can be found in fluid 
dynamics including turbulence modelling, heat transport, oil extraction from underground 
reservoirs, electro-magnetism, semi-conductor devices, environmental and civil 
engineering as well as in chemical and biological sciences. The solutions of 
convection-dominated transport problems are typically characterized by the occurrence 
of sharp moving fronts and interior or boundary layers. The key challenge for the 
accurate numerical approximation of solutions to convection-dominated problems is thus 
the development of discretization schemes with the ability to capture strong 
gradients of solutions without producing spurious oscillations or smearing effects.

A possible remedy is the application of one of the numerous stabilization concepts that
have been proposed and studied for various discretization techniques in the 
recent years;~cf.~\cite{BruchhaeuserRST08}. 
Here we focus on using finite element discretizations along with residual-based 
stabilizations. Among theses techniques, we choose the streamline upwind 
Petrov--Galerkin (SUPG) method \cite{BruchhaeuserHB79,BruchhaeuserBH81}, which 
aims at reducing non-physical oscillations in streamline direction. Besides the class of 
residual-based stabilization techniques, flux-corrected
transport schemes (cf., e.g., \cite{BruchhaeuserKLT12}) have recently been developed 
and investigated strongly. They have shown their potential to handle the characteristics 
of convection-dominated problems and resolve sharp fronts with high 
accuracy;~cf.~\cite{BruchhaeuserJS08}. Recently, numerical analyses of 
these methods were presented; cf.~\cite{BruchhaeuserBJK16}. In contrast to residual-based 
stabilizations, flux-corrected transport schemes aim at a stabilization on the algebraic 
level. In \cite{BruchhaeuserJS08}, a competitive numerical investigation of the 
performance properties of these and further stabilization concepts is given. In 
\cite{BruchhaeuserJS08} and many other works of the literature authors conclude that 
spurious oscillations in the numerical approximation of convection-dominated problems can 
be reduced by state-of-the-art stabilization techniques, but nevertheless the results are 
not satisfactory yet, in particular, if applications of practical interest and in three 
space dimensions are considered.

A further and widespread technique to capture singular phenomena and sharp 
profiles of solutions is the application of adaptive mesh refinement based on an a 
posteriori error control mechanism. For a review of a posteriori error estimation 
techniques for finite element methods and automatic mesh generation we refer, for 
instance, to the monograph \cite{BruchhaeuserV13}.
The design of an adaptive method requires the availability of an appropriate
a posteriori error estimator. For convection-dominated problems, the derivation of such 
an error estimator, that is robust with respect to the small perturbation parameter of 
the partial differential equation, is delicious and has borne out to be a considerable 
source of trouble. Existing a posteriori error estimates are typically 
either non-robust with respect to the perturbation parameter or provide a control 
of quantities that are typically not of interest in practice or a control in 
non-computable error norms; cf.~\cite{BruchhaeuserDEV13,BruchhaeuserJN13}. Consequently, 
the quality of adaptively refined grids is often not satisfactory yet. Further, only a 
few contributions have been published for convection-dominated problems and the 
considered type of stabilized finite element discretizations. For a more detailed 
discussion we refer to \cite{BruchhaeuserJN13}. 

In this work we use an adaptive method that is based on dual weighted residual 
a posteriori error estimation \cite{BruchhaeuserBR98,BruchhaeuserBR01,BruchhaeuserBR03}.
The Dual Weighted Residual method (DWR) aims at the economical
computation of arbitrary quantities of physical interest, point-values or line/surface 
integrals for instance, by properly adapting the computational mesh. 
Thus, the mesh adaptation process can be based on the computation and control of a
physically relevant goal quantity instead of a control in the traditional 
energy- or $L^2$-norm. In particular in the context of convection-dominated transport, 
the control of local quantities is typically of greater importance than the one of global 
quantities arguing for the application of DWR based techniques.  The DWR approach relies 
on a variational formulation of the discrete problem and uses duality techniques to 
provide a rigorous a posteriori error representation from that a computable error 
indicator can be derived. Of course, such an error estimation can also be obtained with 
respect to global quantities and norms, e.g., the $L^2$-norm or energy 
norm; cf.~\cite{BruchhaeuserBGR10}. The exact error representation within the DWR method 
cannot be evaluated directly, since it depends on the unknown exact solution of a so 
called dual or adjoint problem. The dual solution is used for weighting the local 
residuals within the error representation and has to be computed numerically. This 
approximation cannot be done in the finite element space of the primal problem, since it 
would result in an useless vanishing approximation of the error quantity; 
cf.~\cite{BruchhaeuserBR03}. Approximation by higher-order methods, approximation by 
higher-order interpolation, approximation by difference quotients and approximation by 
local residual problems have been considered so far as suitable approaches for the 
approximation of the dual solution; cf.~\cite{BruchhaeuserBR03}. 

In this work we combine the DWR approach with SUPG stabilized approximations of 
convection-dominated problems. Even though the DWR approach has been applied 
to many classes of partial differential equations, our feeling is that its 
potential for the numerical approximation of convection-dominated problems has not been 
completely understood and explored yet. For simplicity, we restrict ourselves to 
stationary convection-dominated problems here. This is done in order to focus on the 
interaction of stabilization and error control in a simplified framework rather than 
considering sophisticated problems. In \cite{BruchhaeuserEW17}, higher-order finite elements and a 
partition-of-unity technique are used to get the local error estimations within the DWR 
method for a class of elliptic problems. Similarly to \cite{BruchhaeuserEW17}, we solve 
the dual problem by using higher-order finite element techniques, which is a key 
ingredient of this work. However, we differ from \cite{BruchhaeuserEW17} with respect to the 
computation of the local error indicator. Here, we follow the classical way of the DWR philosophy, 
receiving the error representation on every mesh element by a cell-wise 
integration by parts.
Due to the specific character of convection-dominated problems our computational 
experience is that the error control needs a particular care in regions with 
interior and boundary layers and in regions with sharp fronts in order to get an 
accurate quantification of the numerical errors. 
This is in contrast to other works of the literature on the DWR method in that 
strong effort is put onto the reduction of the computational costs for solving the dual 
problem. In numerical experiments we will illustrate the high impact of the proper choice 
of the weights and thereby of the dual solution on the mesh adaptation process. 
The key motivation in this work is to reduce sources of inaccuracies and non-sharp 
estimates in the a posteriori error representation as far as possible 
in order to avoid numerical artefacts. Thereby we aim to improve the quality of the 
numerical approximation and error control in particular in regions with sharp fronts 
and sensitive solution profiles where the application of interpolation 
techniques is expected to be highly defective. We note that in the case of nonlinear 
problems, like the Navier--Stokes equations, the computational costs for solving the 
primal problem dominate, since here a Newton or fixed-point iteration has to be applied, 
whereas the dual problem always remains a linear one. Therefore, in the case of 
nonlinear problems a higher order approach for the dual problem does not necessarily 
dominate the overall computational costs. Our approach is elaborated by careful 
numerical investigations in two and three space dimensions that represent a further 
ingredient of this work.

This work is organized as follows. In Sect.~\ref{BruchhaeuserSec:problem} we 
introduce our model problem together with some global assumptions and our 
notation. Further we present a short derivation of the DWR approach as well as
the finite element approximation in space.
We conclude Sect.~\ref{BruchhaeuserSec:problem} by presenting the SUPG 
stabilized form of the discrete scheme. In Sect.~\ref{BruchhaeuserSec:error} 
we derive a localized error representation in terms of a user chosen target 
quantity. In Sect.~\ref{BruchhaeuserSec:practical} some implementational issues 
and our adaptive solution algorithm are addressed. 
Finally, in Sect.\ \ref{BruchhaeuserSec:numerical} the results of numerical 
computations in two and three space dimensions are presented in order to 
illustrate the feasibility and potential of the proposed approach. 

\section{Problem formulation and stabilized discretization}
\label{BruchhaeuserSec:problem}
In this section we first present our model problems. Then we sketch the stabilized 
approximation of the primal and dual problem within the DWR framework.

\subsection{Model problem and variational fomulation}
\label{BruchhaeuserSubSec:model}
In this work we consider the steady linear convection-diffusion-reaction problem 
\begin{equation}
\begin{array}{r@{\;}c@{\;}l@{\hspace*{2ex}}l}
-\nabla\cdot\left(\varepsilon \nabla u\right) + \boldsymbol{b}\cdot\nabla u + \alpha u &= & f 
 & \mbox{in } \Omega \,,\\[1ex]
u &= & 0  & \mbox{on } \partial\Omega\,.
\end{array}
\label{BruchhaeuserEq:cdr}
\end{equation}
We assume that $\Omega\subset \mathbb{R}^d$, with $d=2$ or $d=3$, is a polygonal or 
polyhedral bounded domain with Lipschitz boundary $\partial\Omega\,.$ For
brevity, problem \eqref{BruchhaeuserEq:cdr} is equipped with homogeneous
Dirichlet boundary conditions. In our numerical examples 
in Sect.~\ref{BruchhaeuserSec:numerical}, we also consider other types of boundary 
conditions. In Remark~\ref{BruchhaeuserRem:InhomDirichletBC}, the 
incorporation of nonhomogeneous Dirichlet and Neumann boundary conditions is 
briefly addressed. Problem \eqref{BruchhaeuserEq:cdr} is considered as a prototype model 
for more sophisticated equations of practical interest, for instance, for the  
Navier--Stokes equations of incompressible viscous flow. 
For an application of our approach to semilinear problems with nonlinear 
reactive terms we refer to \cite{BruchhaeuserS14}.

Here, $0 < \varepsilon \ll 1$ is a small positive diffusion coefficient, 
$\boldsymbol{b}\in(H^1(\Omega))^d\cap(L^\infty(\Omega))^d$ is the flow field or 
convection tensor, 
$\alpha\in L^\infty(\Omega)$ is the reaction coefficient, and
$f\in L^2(\Omega)$ is a given outer source of the unknown scalar quantity $u\,.$
Furthermore, we assume 
that the following condition is fulfilled:
 \begin{align}
 \nabla\cdot \boldsymbol{b}(\boldsymbol{x}) = 0 \;\;\text{and}\;\; 
 \alpha(\boldsymbol{x}) \geq 0 & 
 \quad \forall \boldsymbol{x} \in \Omega\,. 
 \label{BruchhaeuserEq:cond1} 
 \end{align}
It is well known that problem \eqref{BruchhaeuserEq:cdr} along with 
condition~\eqref{BruchhaeuserEq:cond1} admits a unique weak solution 
$u\in V=H_0^1:=\big\{v \in H^1(\Omega)\big| \; v|_{\partial\Omega} = 0\big\}$ 
that satisfies the following variational formulation; 
cf.,~e.g.~\cite{BruchhaeuserRST08, BruchhaeuserA95, BruchhaeuserJN13}.

\noindent
\textit{\indent Find $u \in V$ such that}
\begin{equation}
 A(u)(\varphi) = F(\varphi) \quad \forall \varphi \in V\,,
 \label{BruchhaeuserEq:weakp}
\end{equation}
\textit{\indent where the bilinear form $A:V\times V \mapsto \mathbb{R}$ 
and the linear form $F:V\mapsto \mathbb{R}$ are given 
by
}
\begin{eqnarray*}
 A(u)(\varphi) & := & (\varepsilon\nabla u, \nabla \varphi) 
 + (\boldsymbol{b}\cdot \nabla u, \varphi)
 +(\alpha u, \varphi)\,, \\
 F(\varphi) & := & (f,\varphi)\,.
\end{eqnarray*}
We denote by $(\cdot,\cdot)$ the inner product of $L^2(\Omega)$ and by $\|\cdot\|$ 
the associated norm with $\|v\|=(v,v)^{\frac{1}{2}}=
(\int_\Omega |v|^2 d\boldsymbol{x})^{\frac{1}{2}}\,.$

\subsection{The Dual Weighted Residual approach}
\label{BruchhaeuserSubSec:dwr}
The DWR method aims at the control of an error in an arbitrary user-chosen target 
functional $J$ of physical relevance. To get an error representation with respect to this 
target functional, an additional dual problem has to be solved. Before we focus on this 
error representation, we introduce the derivation of the dual problem of 
\eqref{BruchhaeuserEq:weakp} needed below in the DWR approach. For this, we consider the 
Euler-Lagrangian method of constrained optimization. For some given functional 
$J:V\mapsto \mathbb{R}$ we consider solving
\begin{equation*}
J(u) = \min\{J(v)\,, \; v\in V\,, \; \text{where}\; A(v)(\varphi)=F(\varphi)\; \forall 
\varphi \in V\}\,.
\end{equation*}
For this we define the corresponding Lagrangian functional $\mathcal{L}:V \times
V \mapsto \mathbb{R}$ by
\begin{equation}
 \mathcal{L}(u,z):=J(u)+F(z)-A(u)(z)\,,
 \label{BruchhaeuserDef:Lext}
\end{equation}
where we refer to $z\in V$ as the dual variable (or Lagrangian multiplier), 
cf.~\cite{BruchhaeuserBR03}. We determine a stationary point $\{u,z\}\in V\times V$ of 
$\mathcal{L}(\cdot,\cdot)$ by the condition that 
\begin{equation}
 \mathcal{L}'(u,z)(\psi,\varphi) = 0 \quad \forall \{\psi,\varphi\}\in V\times V\,,
 \label{BruchhaeuserEq:stationarycondition}
\end{equation}
or, equivalently, by the system of equations that
\begin{align}
A'(u)(\psi,z) & = J'(u)(\psi) \quad \forall \psi \in V\,, 
\label{BruchhaeuserEq:ucomp}
\\[1ex]
A(u)(\varphi) & = F(\varphi) \quad \quad \;\; \forall 
\varphi \in V\,.
\label{BruchhaeuserEq:zcomp}
\end{align}
Eq.\ \eqref{BruchhaeuserEq:zcomp}, the $z$-component of the stationary 
condition \eqref{BruchhaeuserEq:stationarycondition}, is just the given primal problem 
\eqref{BruchhaeuserEq:weakp} whereas Eq.~\eqref{BruchhaeuserEq:ucomp}, 
the $u$-component of \eqref{BruchhaeuserEq:stationarycondition}, is called the dual or 
adjoint problem. In strong form, the dual problem reads as
\begin{equation}
\begin{array}{r@{\;}c@{\;}l@{\hspace*{2ex}}l}
-\nabla\cdot\left(\varepsilon \nabla z\right) - \boldsymbol{b}\cdot\nabla z + \alpha z &= & j 
 & \mbox{in } \Omega \,,\\[1ex]
z &= & 0  & \mbox{on } \partial\Omega\,, 
\end{array}
\label{BruchhaeuserEq:cdr_dual}
\end{equation}
where $j\in L^2(\Omega)$ is a $L^2$-representation of the target function $J(\cdot)$ 
that is supposed to exist; cf.\ Eq.~\eqref{BruchhaeuserEq:weaklinearformdual}. The 
bilinear form $A'$ is given by
\begin{equation*}
A'(u)(\psi,z) = (\varepsilon\nabla \psi, \nabla z) 
 + (\boldsymbol{b}\cdot \nabla \psi, z)
 +(\alpha \psi, z) = A(\psi)(z)\,.
\end{equation*}
Applying integration by parts to the convective term along with the condition  
\eqref{BruchhaeuserEq:cond1} yields for $A'(u)(\psi,z)$ the representation 
that
\begin{equation}
A^\ast(z)(\psi) := A'(u)(\psi,z) = (\varepsilon\nabla z, \nabla \psi) 
 - (\boldsymbol{b}\cdot \nabla z, \psi)
 +(\alpha z, \psi) \,.
  \label{BruchhaeuserEq:weakbilinearformdual}
\end{equation}
Thus we have the following Euler-Lagrange system.

\textit{Find ${\{u,z\}\in V \times V}$ such that}
\begin{eqnarray}
\label{BruchhaeuserEq:weaklinearformdual_1}
A(u)(\varphi) & = &  \;F(\varphi)   \quad \forall \varphi \in V\,, \\
A^\ast(z)(\psi) & = &  \;J(\psi)  \quad \forall \psi \in V\,, 
\label{BruchhaeuserEq:weaklinearformdual_2}
\end{eqnarray}
\textit{\indent where the functional $J(\cdot)$ is supposed to admit the 
$L^2$-representation}
\begin{equation}
J(\psi) := (j(u),\psi)\,.
\label{BruchhaeuserEq:weaklinearformdual}
\end{equation}

\subsection{Discretization in space}
\label{BruchhaeuserSubSec:discretization}
Here we present the spatial discretization of \eqref{BruchhaeuserEq:cdr}. 
We use Lagrange type finite element spaces of continuous functions that are piecewise 
polynomials. For the discretization in space, we consider a decomposition 
$\mathcal{T}_{h}$
of the domain $\Omega$ into disjoint elements $K$, such that 
$\bar{\Omega}=\cup_{K\in\mathcal{T}_{h}}\bar{K}$. Here, we choose the elements 
$K\in\mathcal{T}_{h}$ to be quadrilaterals for $d=2$ and hexahedrals for $d=3$. We denote 
by $h_{K}$ the diameter of the element $K$. The global space 
discretization parameter $h$ is given by $h:=\max_{K\in\mathcal{T}_{h}}h_{K}$. 
Our mesh adaptation process yields locally refined and coarsened cells, which is 
facilitated by using hanging nodes \cite{BruchhaeuserCO84}. We point out that 
the global conformity of the finite element approach is preserved since the unknowns 
at such hanging nodes are eliminated by interpolation between the neighboring 
'regular' nodes, cf.~\cite{BruchhaeuserBR03}.

The construction of the approximating function spaces is done on the reference 
element $\hat{K}:= [0, 1]^{d}\,,d=2, 3$. Therefore, we introduce the mapping 
$\mathcal{T}_{K}:\hat{K}\rightarrow K$ from the reference element $\hat{K}$ onto
an element $K\in\mathcal{T}_{h}$. On that reference domain $\hat{K}$, we 
introduce the following finite element spaces
\begin{displaymath}
\begin{array}{r@{\,}c@{\,}l}
\hat{\mathcal{Q}}_{h}^{p_{1}, p_{2}} &:=&
\Bigg\{ \hat{\varphi} : [0,1]^{2} \rightarrow \mathbb{R} \,\Bigg|\,
  \hat{\varphi}(\hat{\boldsymbol{x}}) :=
  \displaystyle \sum_{i=0}^{p_{1}} \sum_{j=0}^{p_{2}}
  \hat{\varphi}_{i,j} x_{1}^{i} x_{2}^{j}\,,
  \text{ with }
  \hat{\varphi}_{i,j} \in \mathbb{R} \Bigg\} \quad \text{and}\\[4.5ex]
\hat{\mathcal{Q}}_{h}^{p_{1}, p_{2}, p_{3}} &:=&
  \Bigg\{\hat{\varphi}:[0,1]^{3} \rightarrow\mathbb{R} \,\Bigg|\,
  \hat{\varphi}(\hat{\boldsymbol{x}}):= 
  \displaystyle\sum_{i=0}^{p_{1}}\sum_{j=0}^{p_{2}}\sum_{k=0}^{p_{3}}
  \hat{\varphi}_{i,j,k}x_{1}^{i}x_{2}^{j}x_{3}^{k}\,,
  \text{ with } \hat{\varphi}_{i,j,k}\in\mathbb{R}\Bigg\}\,.
\end{array}
\end{displaymath}
Then, for an arbitrary polynomial degree $p\geq0$, the approximating finite 
element spaces are defined by
\begin{equation}
  V_{h}^{p}:=
  \begin{cases}
  \big\{v\in V\cap C(\bar{\Omega})\big|v|_{K}\circ\mathcal{T}^{-1}_{K}\in
  \hat{\mathcal{Q}}_{h}^{p,p}
  \,,\forall K\in\mathcal{T}_{h} \big\}\,, & d = 2\,,\\[1ex]
  \big\{v\in V\cap C(\bar{\Omega})\big|v|_{K}\circ\mathcal{T}^{-1}_{K}\in
  \hat{\mathcal{Q}}_{h}^{p,p,p}
  \,,\forall K\in\mathcal{T}_{h} \big\}\,, & d = 3\,.
  \end{cases}
\label{Bruchhaeusereq:femspaces}
\end{equation}

\subsection{Streamline upwind Petrov-Galerkin stabilization}
\label{BruchhaeuserSubSec:streamline}
In order to reduce spurious and non-physical oscillations of the discrete 
solutions, arising close to layers or sharp fronts, we apply the SUPG method,
a well-known residual based stabilization technique for finite element 
approximations;  
cf.~\cite{BruchhaeuserRST08, BruchhaeuserAJ15, BruchhaeuserJS08, BruchhaeuserBS12}. 
The SUPG approach aims at an stabilization in the streamline direction; 
cf.\ \cite{BruchhaeuserRST08}. In particular, existing a priori error analysis 
ensure its convergence in the natural norm of the scheme including the control of the 
approximation error in streamline direction; cf.\ \cite[Thm. 3.27]{BruchhaeuserRST08}. 
Applying the SUPG approach to the discrete counterpart of 
\eqref{BruchhaeuserEq:weaklinearformdual_1} and 
\eqref{BruchhaeuserEq:weaklinearformdual_2} yields the following stabilized 
discrete system of equations.

\textit{Find $\{u_h,z_h\}~\in~V_h^p~\times~V_h^{p+s}\,,~s\geq~1\,,$ such that}
\begin{eqnarray}
A_S(u_h)(\varphi_h) & = &  \;F(\varphi_h)  \quad \forall \varphi_h \in V_h^p\,, 
\label{BruchhaeuserEq:ASprimal}\\
A_S^\ast(z_h)(\psi_h) & = &  \;J(\psi_h) \quad \forall \psi_h \in V_h^{p+s}\,, s\geq 1\,,
\label{BruchhaeuserEq:ASdual}
\end{eqnarray}
\textit{\indent where the stabilized bilinear forms are given by}
\begin{eqnarray*}
A_S(u_h)(\varphi_h) & := & A(u_h)(\varphi_h) + S(u_h)(\varphi_h)\,, \\
A_S^\ast(z_h)(\psi_h) & := & A^\ast(z_h)(\psi_h) + S^\ast(z_h)(\psi_h)\,,
\end{eqnarray*}
\textit{\indent and the stabilized terms are defined by}
\begin{eqnarray*}
S(u_h)(\varphi_h) & := & \sum_{K \in \mathcal{T}_h}\delta_K(R(u_h),\boldsymbol{b}\cdot
\nabla \varphi_h)_K   \,, 
\\
R(u_h) & := & -\nabla \cdot (\varepsilon \nabla u_h) + \boldsymbol{b} \cdot \nabla u_h
+\alpha u_h - f\,, \\
S^\ast(z_h)(\psi_h) & := & \sum_{K \in 
\mathcal{T}_h}\delta_K^\ast(R^\ast(z_h),-\boldsymbol{b}\cdot
\nabla \psi_h)_K   \,, 
\\
R^\ast(z_h) & := & -\nabla \cdot (\varepsilon \nabla z_h) - \boldsymbol{b} \cdot \nabla z_h
+ \alpha z_h - j(u_h)\,. 
\end{eqnarray*}

\begin{remark}
The proper choice of the stabilization parameters $\delta_K$ and $\delta_K^\ast$ is an
important issue in the application of the SUPG approach; 
cf.~\cite{BruchhaeuserJKS11,BruchhaeuserJK13} and the discussion therein. 
As proposed by the analysis of stabilized finite element methods in 
\cite{BruchhaeuserLR06,BruchhaeuserBS12}, we choose the parameter $\delta_K$ and 
$\delta_K^\ast$ as 
\begin{equation*}
\delta_K,\delta_K^\ast \sim \text{min}\bigg\{\frac{h_K}{p\|\boldsymbol{b}\|_{L^\infty(K)}};
\frac{h_K^2}{p^4\varepsilon};
\frac{1}{\alpha}
\bigg\}\,.
\end{equation*}
Here, the symbol $\sim$ denotes the equivalence up to a multiplicative constant 
independent of $K$. This constant has to be understood as a numerical tuning parameter.
\end{remark}

\begin{remark}
We note that the stabilization within the dual problem acts in the negative 
direction of the flow field $\boldsymbol{b}$; cf.\ Eq.\ 
\eqref{BruchhaeuserEq:weakbilinearformdual}. The discrete dual probleme 
\eqref{BruchhaeuserEq:ASdual} is based on a \textit{first dualize and then 
stabilize (FDTS)} principle in that the dual problem of the weak equation 
\eqref{BruchhaeuserEq:weaklinearformdual_2} is derived first. The stabilization is then 
implemented by applying the SUPG method to the discrete counterpart of the dual problem. 
The alternative strategy \textit{first stabilize and then dualize (FSTD)} of 
transposing the stabilized fully discrete equation \eqref{BruchhaeuserEq:ASprimal} 
requires differentiation of the stabilization terms. In general, the strategies 
\textit{FDTS} and \textit{FSTD} do not commute with each other. In the literature 
(cf.\ \cite{BruchhaeuserB00}) the possibility of stability problems is noted for the  
 \textit{FSTD} strategy. In our performed numerical experiments the \textit{FSTD} 
strategy did not show any lack of stability but led to slightly weaker results; cf.\ 
\cite{BruchhaeuserS14}. For these reasons we focus on the \textit{FDTS} strategy only in 
this work. 
\end{remark}

\section{Error Estimation}
\label{BruchhaeuserSec:error}

In this section our aim is to derive a localized (i.e.\ elementwise) a posteriori error 
representation for the stabilized finite element approximation in terms of the target 
quantity $J(\cdot)$ by using the concepts of the DWR approach briefly introduced in 
Subsect.~\ref{BruchhaeuserSubSec:dwr}. Afterwards, an adaptive mesh refinement process is 
built upon the thus given error representation. Refining and coarsening of the finite 
element mesh is considered. To achieve this aim, some abstract results are needed. First, 
we show an error representation by means of the Lagrangian functional. From this, an 
error representation in terms of the primal and dual residual is deduced. In the final 
step the dual residual is substituted by the primal residual and some stabilization 
terms. 

To start with, we put 
\begin{eqnarray}
x & := &\{u,z\}\,, \,y:=\{\psi,\varphi\} \in V \times V \,,
\label{BruchhaeuserEq:xandy} \\
x_h & := & \{u_h,z_h\}\,, \,y_h:=\{\psi_h,\varphi_h\} \in V_h^p \times V_h^p \,,
\label{BruchhaeuserEq:xhandyh}
\end{eqnarray}
and
\begin{equation*}
\mathcal{S}(x_h)(y_h):=S(u_h)(\varphi_h) + S^\ast(z_h)(\psi_h)\,.
\end{equation*}
The discrete solution $x_h \in V_h^p \times V_h^p$ then satisfies the variational equation
\begin{equation}
\mathcal{L}'(x_h)(y_h) = \mathcal{S}(x_h)(y_h) \quad \forall y_h \in V_h^p \times V_h^p\,.
\label{BruchhaeuserEq:LstrichS}
\end{equation}
Now we develop the error in terms of the Lagrangian functional.
\begin{theorem}
 Let $X$ be a function space and $\mathcal{L}: X \rightarrow \mathbb{R}\;$ be a
 three times differentiable functional on $X$. Suppose that $x_c\in X_c$ with 
 some (''continuous'') function space 
$X_c \subset X$ is a stationary point of $\mathcal L$. Suppose that $x_d \in X_d$ with 
some (''discrete'') function space $X_d \subset X\,,$ with not necessarily $X_d \subset 
X_c\,,$ is a Galerkin approximation to $x_c$ being defined by the equation 
\begin{equation}
\label{BruchhaeuserDef:L0}
\mathcal{L}'(x_d)(y_d) = \mathcal{S}(x_d)(y_d) \quad  \forall y_d \in X_d\,.
\end{equation}
In addition, suppose that the auxiliary condition 
\begin{equation}
\label{BruchhaeuserDef:L1}
\mathcal{L}'(x_c)(x_d) = 0 
\end{equation}
is satisfied. Then there holds the error representation
\begin{equation*}
\mathcal{L}(x_c) - \mathcal{L}(x_d) = \frac{1}{2} \mathcal{L}'(x_d) (x_c- y_d) + 
\frac{1}{2} 
\mathcal{S}(x_d)(y_d -x_d) + \mathcal{R} \quad \forall y_d \in X_d\,,
\end{equation*}
where the remainder $\mathcal{R}$ is defined by
\begin{equation}
\label{BruchhaeuserDef:L4}
\mathcal{R} = \frac{1}{2}\int_0^1 \mathcal{L}^{\prime\prime\prime}(x_d +s e)(e,e,e)
\cdot s \cdot (s-1)\,\mathrm{d} s\,,
\end{equation}
with the notation $e:=x_c-x_d$.
\label{BruchhaeuserThm:L}
\end{theorem}

\begin{proof} 
We let $e = x_c -x_d$. By the fundamental theorem of calculus it holds that 
\[
\mathcal L(x_c) - \mathcal L(x_d) = \int_0^1 \mathcal L'(x_d+s e)(e)\,\mathrm{d} s\,.
\]
Approximating the integral by the trapezoidal rule yields that 
\begin{equation}
\label{BruchhaeuserEq:EL10}
\mathcal{L}(x_c) - \mathcal{L}(x_d) = \frac{1}{2} \mathcal{L}'(x_d)(x_c -x_d) + 
\frac{1}{2} 
\mathcal{L}'(x_c)(x_c -x_d) + \mathcal{R}\,,
\end{equation}
with $\mathcal{R}$ being defined by 
\eqref{BruchhaeuserDef:L4}
. By the supposed stationarity 
of $\mathcal{L}$ in $x_c$ along with the assumption 
\eqref{BruchhaeuserDef:L1}
the second of the terms on the right-hand side of 
\eqref{BruchhaeuserEq:EL10} 
vanishes. Together with 
eq.\ 
\eqref{BruchhaeuserDef:L0} 
we then 
get that 
\begin{eqnarray*}
\mathcal{L}(x_c) - \mathcal{L}(x_d) & = & \frac{1}{2}\mathcal{L}'(x_d) (x_c- y_d) 
+ \frac{1}{2} \mathcal{L}'(x_d) (y_d - x_d) + \mathcal{R}\\[1ex]
& = & \frac{1}{2} \mathcal{L}'(x_d) (x_c- y_d) + \frac{1}{2} 
\mathcal{S}(x_d)(y_d-x_d) + \mathcal{R}\,,
\end{eqnarray*}
for all $y_d\in X_d$. This completes the proof of the theorem.
\end{proof}
For the subsequent theorem we introduce the primal and dual residuals by
\begin{eqnarray}
\rho(u_h)(\varphi) & := &  \;F(\varphi)-A(u_h)(\varphi)   
\;\quad\quad\quad \forall \varphi \in V\,, 
\label{BruchhaeuserDef:primalresidual}\\
\rho^\ast(z_h)(\psi) & := &  \;J'(u_h)(\psi)-A^\ast(z_h)(\psi)   
\quad \forall \psi \in V\,.
\label{BruchhaeuserDef:dualresidual}
\end{eqnarray}
\begin{theorem}
\label{BruchhaeuserThm:J}
Suppose that $\{u,z\}\in V \times V$ is a stationary point of 
the Lagrangian functional $\mathcal{L}$ defined in \eqref{BruchhaeuserDef:Lext} 
such that \eqref{BruchhaeuserEq:stationarycondition} is satisfied. 
Let $\{u_{h},z_{h}\}\in V_h^p \times V_h^p$ denote its Galerkin approximation 
being defined by \eqref{BruchhaeuserEq:ASprimal} and \eqref{BruchhaeuserEq:ASdual} 
such that \eqref{BruchhaeuserEq:LstrichS} is satisfied. 
Then there holds the error representation that 
\begin{equation}
J(u) - J(u_{h}) = \frac{1}{2}\rho(u_{h}) (z - \varphi_{h}) + 
\frac{1}{2}\rho^\ast(z_{h}) (u-\psi_{h}) + \mathcal{R}_{\mathcal{S}} + 
\mathcal{R}_{J}\,,
\label{BruchhaeuserThm:J1}
\end{equation}
for arbitrary functions $\{\varphi_{h},\psi_{h}\}\in V_h^p \times V_h^p$, where 
the remainder terms are defined by 
\begin{equation*}
\mathcal{R}_{\mathcal{S}} := \frac{1}{2} S(u_{h})(\varphi_{h}+z_{h}) + \frac{1}{2} 
S^\ast(z_{h})(\psi_{h}-u_{h})\,,
\end{equation*}
and
\begin{equation}
\label{BruchhaeuserThm:J3}
\mathcal{R}_{J} := \frac{1}{2}\int_0^1 J'''(u_{h}+s\cdot e)(e,e,e) 
\cdot s\cdot (s-1)\,\mathrm{d} s\,,
\end{equation}
with $e=u-u_{h}$.
\end{theorem}

\begin{proof}
Let $x$, defined by \eqref{BruchhaeuserEq:xandy}, be a 
stationary point of $\mathcal{L}$ in \eqref{BruchhaeuserDef:Lext} such that 
\eqref{BruchhaeuserEq:stationarycondition} is satisfied. 
Let $x_{h}$, defined by \eqref{BruchhaeuserEq:xhandyh}, denote the Galerkin 
approximation of $x$ that is given by \eqref{BruchhaeuserEq:ASprimal} and 
\eqref{BruchhaeuserEq:ASdual}, respectively.  
From \eqref{BruchhaeuserDef:Lext} along with 
\eqref{BruchhaeuserEq:stationarycondition} and \eqref{BruchhaeuserEq:ASprimal} 
we conclude that 
\begin{equation*}
\label{BruchhaeuserThm:J4}
J(u) - J(u_{h}) = \mathcal{L}(x) - \mathcal{L}(x_{h}) + 
S(u_{h})(z_{h})\,.
\end{equation*}
Using Thm.~\ref{BruchhaeuserThm:L} we get that 
\begin{equation}
\label{BruchhaeuserThm:J5}
J(u) - J(u_{h}) = \frac{1}{2} \mathcal{L}'(x_{h})(x-y_{h}) + 
\frac{1}{2} \mathcal{S}(x_{h})(y_{h} -x_{h}) + S(u_{h})(z_{h})+ \mathcal{R}\,,
\end{equation}
for all $\{\psi_{h},\varphi_{h}\}\in V_h^p \times V_h^p$ with the remainder 
$\mathcal{R}$ being defined by \eqref{BruchhaeuserDef:L4}. Recalling 
the definition \eqref{BruchhaeuserDef:Lext} of $\mathcal{L}$ yields for the 
remainder $\mathcal{R}$ the asserted representation \eqref{BruchhaeuserThm:J3}, 
due to the fact that all parts of the third derivative of $\mathcal{L}$ vanish 
except the third derivative of $J$.

Taking into account that 
$\mathcal{L}_u(u,z)(\psi) + \mathcal{L}_z(u,z)(\varphi) = 0$,where 
$\mathcal{L}_u$ and $\mathcal{L}_z$ denote the $u$-component and $z$-component, 
respectively, of the stationary condition 
\eqref{BruchhaeuserEq:stationarycondition}, we can write the Fr\'{e}chet
derivative of the Lagrangian functional as
\begin{eqnarray*}
\mathcal{L}'(u_{h},z_{h})(u-\psi_{h},z- \varphi_{h}) 
& = & \; J'(u_{h})(u-\psi_{h}) - A^{\ast}(z_{h})(u-\psi_{h})  \\
& & + \; F(z-\varphi_{h}) -A(u_{h})(z-\varphi_{h})  \\
& = & \; \rho^\ast(z_{h})(u-\psi_{h}) + \rho(u_{h})(z-\varphi_{h})\,,
\end{eqnarray*}
for all $\{\psi_{h},\varphi_{h}\}\in V_h^p \times V_h^p$. 
Substituting this identity into \eqref{BruchhaeuserThm:J5} yields that
\begin{equation}
\label{BruchhaeuserThm:J6}
\begin{split}
J(u) - J(u_{h})  =& \frac{1}{2} 
\rho^\ast(z_{h})(u-\psi_{h}) + \frac{1}{2}
\rho(u_{h})(z-\varphi_{h})\\
& + \frac{1}{2} \mathcal{S}(x_{h})(y_{h} -x_{h}) + S(u_{h})(z_{h})+ 
\mathcal{R}_{J}\,.
\end{split}
\end{equation}
Finally, we note that 
\begin{eqnarray}
\nonumber
& & \frac{1}{2} \mathcal{S}(x_{h})(y_{h} -x_{h}) + 
S(u_{h})(z_{h})\\[1ex]
\nonumber
& & = \frac{1}{2} S(u_{h})(\varphi_{h}-z_{h}) + \frac{1}{2} 
S^\ast(z_{h})(\psi_{h}-u_{h}) 
+ S(u_{h})(z_{h})\\[1.5ex] 
& & = \frac{1}{2} S(u_{h})(\varphi_{h}+z_{h}) + \frac{1}{2} 
S^\ast(z_{h})(\psi_{h}-u_{h})
\,.
\label{BruchhaeuserThm:J7}
\end{eqnarray}
Combining \eqref{BruchhaeuserThm:J6} with \eqref{BruchhaeuserThm:J7} proves the 
assertion of the theorem. 
\end{proof}
In the error respresentation \eqref{BruchhaeuserThm:J1} the continuous solution 
$u$ is required for the evaluation of the adjoint residual. In the following theorem we 
show that the adjoint residual coincides with the primal residual up to a quadratic 
remainder. This observation will be used below to find our final error respresentation in 
terms of the goal quantity $J$ and a suitable linearization for its computational 
evaluation or approximation, respectively.  
\begin{theorem}
\label{BruchhaeuserThm:ResDev}
Under the assumptions of Thm.~\ref{BruchhaeuserThm:J}, and with the definitions 
\eqref{BruchhaeuserDef:primalresidual} and \eqref{BruchhaeuserDef:dualresidual} of the 
primal and dual residual, respectively, there holds that 
\begin{equation*}
\begin{split}
\rho^\ast(z_{h}) (u-\psi_{h}) & = \rho(u_{h}) (z - \varphi_{h}) + 
S(u_{h})(\varphi_{h}-z_{h}) 
+ S^\ast (z_{h})(u_{h}-\psi_{h}) + \Delta \rho_{J}\,,
\end{split}
\end{equation*}
for all $\{\psi_{h},\varphi_{h}\}\in V_h^p \times V_h^p$, where the remainder term is 
given by
\begin{equation}
\label{BruchhaeuserThm:ResDev2}
\Delta \rho_{J}:=  - \int_0^1 J''(u_{h} + s \cdot e)(e,e) \,\mathrm{d} s\,,
\end{equation}
with $e:=u - u_{h}$.
\end{theorem}

\begin{proof}
Let $e:=u - u_{h}$ and $e^\ast:=z -z_{h}$ denote the primal and dual error, 
respectively. For arbitrary $\psi_{h}\in V_h^p$ we put
\begin{equation*}
k(s) := J'(u_{h} + s \cdot e)(u-\psi_{h})- A^\ast(z_{h}+s 
\cdot e^\ast)(u-\psi_{h}) \,.
\end{equation*}
We have that 
\begin{equation*}
k(1) :=J'(u)(u-\psi_{h}) - A^\ast(z)(u-\psi_{h}) = 0\,.
\end{equation*}
From \eqref{BruchhaeuserDef:dualresidual} we get that
\begin{equation*}
k(0) = J'(u_{h})(u - \psi_{h}) - A^\ast(z_{h})(u-\psi_{h})
= \rho^\ast (z_{h}) (u-\psi_{h})\,. 
\end{equation*}
Further, we conclude that 
\begin{equation*}
k'(s) =  J''(u_{h} + s\cdot e)(e,u-\psi_{h}) - A^\ast(e^\ast)(u-\psi_{h})\,.
\end{equation*}
Using \eqref{BruchhaeuserEq:ASdual} and \eqref{BruchhaeuserDef:dualresidual} we 
find that 
\begin{eqnarray}
\nonumber
 \rho^\ast (z_{h}) (u-\psi_{h})  &=& \; J'(u_{h})(u - \psi_{h}) - 
A^\ast(z_{h})(u-\psi_{h}) + S^\ast (z_{h})(\psi_{h}) - S^\ast (z_{h})(\psi_{h}) 
\\[1ex]
\nonumber
& & - J'(u_{h})(u_{h}) + A^\ast(z_{h})(u_{h}) + 
S^\ast(z_{h})(u_{h}) \\[1.5ex]\nonumber 
&=& \;  \rho^\ast (z_{h}) (u-u_{h}) +  S^\ast(z_{h})(u_{h}-\psi_{h})\\[1ex]
\label{BruchhaeuserThm:ResDev5}
&=&  \; \rho^\ast (z_{h}) (e) +  S^\ast (z_{h})(u_{h}-\psi_{h})\,.
\end{eqnarray}
From \eqref{BruchhaeuserThm:ResDev5} along with the theorem of calculus 
$\int\limits_0^1 k'(s) \,\mathrm{d} s = k(1)-k(0)$ it follows that 
\begin{eqnarray}
\nonumber 
 \rho^\ast (z_{h})(u-\psi_{h}) &=& \rho^\ast (z_{h})(e)  + S^\ast 
(z_{h})(u_{h}-\psi_{h})
\\[1ex]\nonumber
&=& \; k(0)-k(1)+ S^\ast (z_{h})(u_{h}-\psi_{h}) \\[1ex] 
\nonumber
&=& \; \int_0^1 \Big(A^\ast(e^\ast)(e) + 
 - J''(u_{h} + s \cdot e)(e,e) 
\Big)\,\mathrm{d}s+ S^\ast(z_{h})(u_{h}-\psi_{h})\\[1ex]
\label{BruchhaeuserThm:ResDev6}
&=& \; A^\ast(e^\ast)(e)  
+ S^\ast(z_{h})(u_{h}-\psi_{h}) + \Delta \rho_{J}\,.
\end{eqnarray}
Next, for the first term on the right-hand side of 
\eqref{BruchhaeuserThm:ResDev6} we get that 
\begin{eqnarray}
\nonumber
A^\ast(e^\ast)(e) &=& 
\; (\varepsilon \nabla e^\ast, \nabla e) -(\boldsymbol{b}\cdot \nabla e^\ast,e)
+(\alpha e^\ast, e) 
 \\[1ex]\nonumber
&=& \; (\varepsilon \nabla e, \nabla e^\ast)+(\boldsymbol{b}\cdot \nabla e, e^\ast) 
+ (\alpha e, e^\ast)\\[1ex]\nonumber
&=& \; F(e^\ast)  - A(u_{h})(e^\ast) \\[1ex]
\label{BruchhaeuserThm:ResDev7}
&=& \; \rho(u_{h})(z-z_{h}) = \rho(u_{h})(z-\varphi_{h}) + 
S(u_{h})(\varphi_{h}-z_{h})\,,
\end{eqnarray}
for all $\varphi_{h}\in V_h^p$. Combining \eqref{BruchhaeuserThm:ResDev6} with 
\eqref{BruchhaeuserThm:ResDev7} yields that
\[
\begin{split}
\rho^\ast (z_{h})(u-\psi_{h}) & = \rho(u_{h})(z-\varphi_{h}) + 
S(u_{h})(\varphi_{h}-z_{h})
+ S^\ast (z_{h})(u_{h}-\psi_{h}) + \Delta \rho_{J}\,,
\end{split}
\]
for all $\{\psi_{h},\varphi_{h}\}\in V_h^p \times V_h^p$ with 
$\Delta \rho_{J}$ being defined by 
\eqref{BruchhaeuserThm:ResDev2}. This proves the assertion of the theorem.
\end{proof}
We summarize the results of the previous two theorems in the following corollary.
\begin{corollary}
\label{BruchhaeuserCor:J}
Under the assumptions of Thm.~\ref{BruchhaeuserThm:J} with the 
definitions \eqref{BruchhaeuserDef:primalresidual} and 
\eqref{BruchhaeuserDef:dualresidual} of the primal and dual residual, respectively, there 
holds the error representation that
 \begin{equation}
\label{BruchhaeuserCor:J1}
J(u) - J(u_{h}) = \rho(u_{h}) (z - \varphi_{h}) + S(u_{h})(\varphi_{h}) + 
\mathcal{R}_{J} + \frac{1}{2} \Delta \rho_{J}\,,
\end{equation}
for arbitrary functions $\varphi_{h} \in V_h^p$, where the remainder term 
$\mathcal{R}_{J}$ is given by \eqref{BruchhaeuserThm:J3} and the linearization 
error $\Delta \rho_{J}$ is defined by \eqref{BruchhaeuserThm:ResDev2}. 
\end{corollary}
In the final step we derive a localized or elementwise approximation of the 
error that is then used for the design of the adaptive algorithm.

\begin{theorem}[Localized error representation]
\label{BruchhaeuserThm:LocPresent}
Let the assumptions of Thm.~\ref{BruchhaeuserThm:J} be satisfied.
Neglecting the higher order error terms in \eqref{BruchhaeuserCor:J1}, 
then there holds as a linear approximation the cell-wise error representation
\begin{equation}
\label{BruchhaeuserEq:localER}
J(u) - J(u_{h})  \doteq   \sum\limits_{K\in\mathcal{T}_h} \Big\{ 
\big(\mathcal{R}(u_{h}), z-\varphi_{h}\big)_{K} 
- \delta_K\big(\mathcal{R}(u_{h}), \boldsymbol{b}\cdot\nabla\varphi_{h}\big)_{K} 
- \big(\mathcal{E}(u_{h}), z-\varphi_{h}\big)_{\partial K}\Big\}\,.
\end{equation}
The cell- and edge-wise residuals are defined by 
\begin{eqnarray}
\label{Bruchhaeusereq:42} \mathcal{R}(u_{h})_{|K} &:=& f + \nabla\cdot(\varepsilon\nabla 
u_{h})
- \boldsymbol{b}\cdot\nabla u_{h} -\alpha u_{h} \,,\\[0.5ex]
\label{Bruchhaeusereq:43} \mathcal{E}(u_{h})_{|\Gamma} &:=& \left\{ \begin{array}{cl} 
\frac{1}{2}\boldsymbol{n}\cdot[\varepsilon\nabla u_{h}] & 
\mbox{ if } \Gamma\subset\partial K\backslash\partial\Omega\,, \\[0.5ex] 
0 & 
\mbox{ if } \Gamma\subset\partial\Omega\,,\\ \end{array}\right.
\end{eqnarray}
where $[\nabla u_{h}]:= \nabla u_{h}{}_{|\Gamma\cap K}
-\nabla u_{h}{}_{|\Gamma\cap K'}$ defines the jump of $\nabla u_{h}$ 
over the inner edges $\Gamma$ with normal unit vector 
$\boldsymbol{n}$ pointing from $K$ to $K'$. 
\end{theorem}

\begin{proof} 
The assertion directly follows from \eqref{BruchhaeuserCor:J1} by neglecting the 
higher order remainder terms $\mathcal R_{J}$ and $\Delta \rho_{J}$ as well as 
applying integration by parts on each cell $K\in \mathcal{T}_h$ to the diffusion  
term in the primal residual \eqref{BruchhaeuserDef:primalresidual}. 
\end{proof}

\begin{remark}
 \label{BruchhaeuserRem:InhomDirichletBC} (Nonhomogeneous Dirichlet and Neumann boundary 
conditions)
We briefly address the incorporation of further types of boundary conditions. First, we 
consider problem \eqref{BruchhaeuserEq:cdr} equipped with the nonhomogeneous 
Dirichlet condition
 \begin{equation*}
  u=g_D \;\; \text{on} \;\; \partial\Omega
 \end{equation*}
for a given function $g\in H^{\frac{1}{2}}(\partial\Omega)\,.$ For this, let 
$\tilde{g}_D\in H^1(\Omega)$ be an extension of $g_D$ in the sense that the trace 
of $\tilde{g}_D$ equals $g_D$ on $\partial\Omega$. Further, let the 
discrete function $\tilde{g}_{D,h}$ be an appropriate finite element approximation of the 
extension $\tilde{g}_D$. Then, the trace on $\partial\Omega$ 
of $\tilde{g}_{D,h}$ represents a discretization of $g_D$. For instance, a nodal 
interpolation of $g_D$ and an extension in the finite element space can be used. This 
allows us to recast the weak form of problem \eqref{BruchhaeuserEq:cdr} and its 
discrete counterpart in terms of $w=u-\tilde{g}_D\in H^1_0(\Omega)$ and $w_h = u_h 
-\tilde{g}_{D,h}\in V_h^p \subset H^1_0(\Omega)$. The previous calculations and 
the derivation of the a posteriori error estimator are then done for the weak problem and 
its discrete counterpart rewritten in terms of $w$ and $w_h$. This yields the result that 
\begin{equation}
\label{BruchhaeuserEq:localERNonhom}
\begin{array}{r@{}l}
J(u) - J(u_{h})  \doteq &{}  \displaystyle\sum\limits_{K\in\mathcal{T}_h} \Big\{ 
\big(\mathcal{R}(u_{h}), z-\varphi_{h}\big)_{K} 
- \delta_K\big(\mathcal{R}(u_{h}), \boldsymbol{b}\cdot\nabla\varphi_{h}\big)_{K} \\[1ex]
&{} - \big(\mathcal{E}(u_{h}), z-\varphi_{h}\big)_{\partial K}\Big\}
-\big((g_D-\tilde{g}_{D,h}),\varepsilon\nabla z\cdot\boldsymbol{n}\big)_{\partial\Omega}\,,
\end{array}
\end{equation}
where $\mathcal{R}(u_{h})$ and $\mathcal{E}(u_{h})$ are given by 
\eqref{Bruchhaeusereq:42} and \eqref{Bruchhaeusereq:43}, respectively. The result of 
Thm.~\ref{BruchhaeuserThm:LocPresent} is thus extended by the last term of 
\eqref{BruchhaeuserEq:localERNonhom}. If a homogeneous Neumann condition is prescribed on 
a part $\partial \Omega_N$ of the boundary $\partial \Omega =\partial \Omega_D \cup 
\partial \Omega_N$, with Dirichlet part $\partial \Omega_D$, then the derivation has to 
be done analogously for the solution space $V=\{v\in H^1(\Omega) \mid v=0 \text{ on } 
\Gamma_D\}$ and its discrete counterpart and the resulting variational problems.
\end{remark}

\section{Practical Aspects}
\label{BruchhaeuserSec:practical}

In this section we present practical aspects regarding the application of the 
result given in  Thm.~\ref{BruchhaeuserThm:LocPresent} in computational studies of 
convection-dominated problems. In particular, we present our mesh adaptation 
strategy that is based on \eqref{BruchhaeuserEq:localER}. We note that the concepts 
described here can be generalized to nonstationary problems; cf.~\cite{BruchhaeuserKB17}.

The error representation \eqref{BruchhaeuserEq:localER}, rewritten as 
\begin{eqnarray}
\nonumber
J(u) - J(u_{h}) & \doteq & \sum\limits_{K\in\mathcal{T}_h} \Big\{ 
\big(\mathcal{R}(u_{h}), z-\varphi_{h}\big)_{K} 
- \delta_K\big(\mathcal{R}(u_{h}), \boldsymbol{b}\cdot\nabla\varphi_{h}\big)_{K} 
- \big(\mathcal{E}(u_{h}), z-\varphi_{h}\big)_{\partial K}\Big\} \\
\nonumber
& &  \\
& = & \eta :=\sum\limits_{K\in\mathcal{T}_h}\eta_K\,,
\label{BruchhaeuserEq:DefEta}
\end{eqnarray}
depends on the discrete primal solution~$u_h$ as well as on the exact dual 
solution~$z$. For the application  of \eqref{BruchhaeuserEq:DefEta} in computations, the 
unknown dual solution~$z$ has to be approximated which results in an approximate error 
indicator $\tilde{\eta}$. As noted before, and as a consequence of 
\eqref{BruchhaeuserCor:J1} and \eqref{BruchhaeuserEq:localER}, respectively, the 
approximation of the dual solution cannot be done in the same finite element space as 
used for the primal problem, since this would result in an useless vanishing error 
representation $\tilde{\eta}=0$, 
due to the Galerkin orthogonality. As a key ingredient of this work, for the numerical 
approximation of the dual solution $z\in V$ we use a finite element approach 
that is of higher polynomial order than the one of the discretization of the primal 
problem. Again, to overcome the difficulties associated with the convection-dominance  
the SUPG method \eqref{BruchhaeuserEq:ASdual} is here applied to the discrete dual 
problem 
as well. Thus, we determine an approximation $z_h$ of the the dual solution $z\in V$ in 
the space $V_h^{p+s}$ for some $s \geq 1$. In contrast 
to many other works of the 
literature we thus use a higher order approach for the approximation of the dual 
solution, i.e.\ $z_h \in V_h^{p+s}$ with $s\geq 1$ whereas $u_h\in V_h^p$, which however 
leads to higher computational costs; cf.~\cite{BruchhaeuserBK15, BruchhaeuserK15} for 
algorithmic formulations and analyses. In the literature, the application of higher 
order interpolation instead of usage of higher order finite element spaces is often 
suggested for the DWR approach; cf.~\cite{BruchhaeuserBR03, BruchhaeuserBGR10}. For 
convection-dominated problems such an interpolation might be defective
and lead to tremendous errors close to sharp layers and fronts. Higher order
techniques show more stability and reduce spurious oscillations 
(cf.~\cite{BruchhaeuserBS12}) which is our key motivation for using a higher order 
approach for the approximation of the dual solution. We also refer to our remark in 
Sect.~\ref{BruchhaeuserSec:Intro} regarding the computational costs if the DWR is applied 
to nonlinear problems as they typically arise in applications of practical interest. 

In order to define the localized error contributions $\tilde{\eta}_K$ we 
consider a hierarchy of sequentially refined meshes $\mathcal{M}_i$, with 
$i \geq 1$ indexing the hierarchy. The corresponding finite element spaces are 
denoted by $V_h^{p+s,i}\,, s\geq 1\,,$ (cf. \eqref{Bruchhaeusereq:femspaces}) 
with the additional index $i$ denoting the mesh hierarchy. We calculate the 
cell-wise contributions to the linearized error representation 
\eqref{BruchhaeuserEq:DefEta} and \eqref{BruchhaeuserEq:localER}, respectively, 
by means of
\begin{equation}
\label{BruchhaeuserEq:DefEtaKM}
\tilde{\eta}_K = 
\; \big(\mathcal{R}(u_{h}^{i}), z_h^{i}-\mathcal{I}_h 
z_h^{i}\big)_K  - \delta_K\big(\mathcal{R}(u_{h}^{i}), 
\boldsymbol{b}\cdot\nabla\mathcal{I}_h
z_h^{i}\big)_K
- \big(\mathcal{E}(u_{h}^{i}), z_h^{i}-\mathcal{I}_h 
z_h^{i}\big)_{\partial K} \,,
\end{equation}
where the cell and edge residuals are given in \eqref{Bruchhaeusereq:42} and 
\eqref{Bruchhaeusereq:43}, respectively. By $\mathcal{I}_h z_h^{i} \in V_h^{p,i}$ we 
denote the nodal based Lagrange interpolation of the higher order approximation 
$z_h^{i} \in V_h^{p+s,i}$, $s\geq 1$ into the lower order finite element space $V_h^p$. Our 
adaptive mesh refinement algorithm based on \eqref{BruchhaeuserEq:DefEtaKM} is summarized 
in the following. 

  \begin{center}
   \textbf{Adaptive solution algorithm (Refining and Coarsening)}
  \end{center}
  \textbf{Initialization} Set $i = 0$ and generate the initial finite element
  spaces for the primal and dual problem. 
  \begin{enumerate}
   \item Solve the \textbf{primal} problem:
   Find $u^i_h \in V^{p,i}_h$ such that 
   \begin{equation*}
   \hspace{-1.5cm} A_S(u^i_h)(\varphi_h) = F(\varphi_h)  \hspace{1.2cm} \forall 
\varphi_h\in V^{p,i}_h \,.
   \end{equation*}
   \item Solve the \textbf{dual} problem:
   Find $z^i_h \in V^{p+s,i}_h\supset V^{p,i}_h\,, s\geq 1\,,$ such that 
   \begin{equation*}
   A^{\ast}_S(z^i_h)(\psi_h) =  J(\psi_h) 
  \hspace{1.2cm} \forall \psi_h\in V^{p+s,i}_h\,, s\geq 1\,.
   \end{equation*}
   Here, $V^{p+s,i}_h$ denotes the finite element space of piecewise polynomials of 
higher order on the mesh $\mathcal{M}_i$.
   \item Evaluate the \textbf{a posteriori} error indicator
   \begin{eqnarray*}
   \tilde{\eta} & := & \sum_{K\in\mathcal{T}_h}
   \tilde{\eta}_K\,,\\
   \tilde{\eta}_K & = & \big( \mathcal{R}(u^i_H),z^i_h-\mathcal{I}_h z^i_h \big)_K
   - \delta_K \big( \mathcal{R}(u^i_H),\boldsymbol{b}\cdot \nabla\mathcal{I}_h z^i_h 
\big)_K 
- \big( \mathcal{E}(u^i_H),z^i_h-\mathcal{I}_h z^i_h \big)_{\partial K} \,,
   \end{eqnarray*}
  where the cell and edge residuals are given in \eqref{Bruchhaeusereq:42} and
  \eqref{Bruchhaeusereq:43}. By $u^i_H$ we denote the nodal based Lagrange interpolation 
of $u^i_h$ in $V_h^{p+s,i}$. Further, $z^i_h$ is the computed dual solution and 
$\mathcal{I}_h z^i_h$ is the interpolation of $z^i_h$ in the finite element 
space $V_h^{p,i}$ of the primal problem. 
  \item Histogram based refinement strategy:
  
  \textbf{Choose} $\theta \in (0.25,5)$. \textbf{Put}
  $\tilde{\eta}_{\text{max}}=\displaystyle\max_{K\in\mathcal{T}_h} |\tilde{\eta}_K|$ 
  and 
  \[ \mu = \theta\frac{\displaystyle\sum_{K\in\mathcal{T}_h}|\tilde{\eta}_K|}{\displaystyle{\# 
K}}\,.\] \\
  \textbf{While} $\mu > \tilde{\eta}_{\text{max}}$: 
  \begin{equation*}
   \mu := \frac{\mu}{2}\,.
  \end{equation*}
  Mark the elements $\tilde{K}$ with $|\tilde{\eta}_{\tilde{K}}| > \mu$ to be refined 
  and those two percent of the elements $\hat{K}$ that provide the smallest 
  contribution to $\tilde{\eta}$ to be coarsened. Generate a new mesh 
  $\mathcal{M}_{i+1}$ by regular coarsening and refinement.
  \item Check the stopping condition:
  
  If $\tilde{\eta}_{\text{max}} < \text{tol}$ or $\tilde{\eta} < \text{tol}$ is satisfied, then the  
  adaptive solution algorithm is terminated; Else, $i$ is increased to $i+1$ and it is 
jumped back to Step 1.
  \end{enumerate}
  
\begin{remark} Regarding the choice of the numerical tuning parameter $\theta$ in Step~4
of the previous algorithm we made the computational experience that a value of 
$\theta$ between 0.25 and 5 typically leads to good results. Further, we note that the 
performance properties of adaptive algorithms are strongly affected by the marking 
strategy. Here, marking is implemented in Step 4. 
We carefully analyzed the cell-wise distribution of the magnitude of the error indicators 
defined in Step 3 of the previous algorithm. The presented histogram based remeshing 
strategy of Step 4 yielded the best results. The so called \emph{D\"orfler}
marking (cf. \cite{BruchhaeuserD96}) or the marking of the largest local error 
indicators represent further popular marking strategies. For a further discussion of 
this issue we refer to, e.g.,~\cite{BruchhaeuserBR03}. 
\end{remark}
\begin{remark} According to the adaptive solution algorithm 
presented above, we use the same mesh for solving the primal and dual problem,
more precisely we use the same triangulation for both problems, but different 
polynomial degrees for the underlying shape functions of the respective finite
element space.
\end{remark}

For measuring the accuracy of the error estimator, we will 
study in our numerical experiments the effectivity index 
\begin{equation}
\label{BruchhaeuserEq:Ieff}
\mathcal{I}_{\mathrm{eff}} = \left|\frac{\tilde{\eta}}{J(u)-J(u_{h})}\right|
\end{equation}
as the ratio of the estimated error $\tilde{\eta}$ of 
\eqref{BruchhaeuserEq:DefEta} over the exact error. 
Desirably, the index $\mathcal{I}_{\mathrm{eff}}$ should be close to one.

\section{Numerical studies}
\label{BruchhaeuserSec:numerical}

In this section we illustrate and investigate the performance properties of
the proposed approach of combining the Dual Weighted Residual method with 
stabilized finite element approximations of convection-dominated problems.
We demonstrate the potential of the DWR method with regard to resolving solution 
profiles admitting sharp layers as they arise in convection-dominated problems. Further 
we investigate the mesh adaptation processes by prescribing various target functionals or 
goal quantities, respectively. For this, standard benchmark problems of the literature 
for studying the approximation of convection-dominated transport are applied. For 
the implementation and our numerical computations we use our \texttt{DTM++} frontend  
software \cite[Chapter 4]{BruchhaeuserK15} that is based on the open source finite 
element library \texttt{deal.II};  cf.~\cite{BruchhaeuserABDHHKMPTW,BruchhaeuserBHK07}. 

\subsection{Example 1 (Hump with circularly layer, 2d)}
\label{BruchhaeuserExample1}

In the first numerical experiment we focus on studying the accuracy of our error 
estimator and the impact of approximating the weights of the dual solution within 
the error indicators \eqref{BruchhaeuserEq:DefEtaKM}. For this, we consider different 
combinations of polynomial orders for the finite element spaces of the primal and dual 
solution. We study problem \eqref{BruchhaeuserEq:cdr} with the prescribed solution
(cf.~\cite{BruchhaeuserJS08,BruchhaeuserBS12,BruchhaeuserAJ15})
\begin{equation}
\label{BruchhaeuserEq:movghump}
u(\boldsymbol{x}) = 
16x_1(1-x_1)x_2(1-x_2)
\cdot \bigg\{\frac{1}{2} + 
\dfrac{\arctan
\big( 
2\varepsilon^{-1/2}\big[r_0^2 - (x_1-x_1^0)^2 - (x_2-x_2^0)^2\big] 
\big)}{\pi} \bigg\}\,.
\end{equation}
where $\Omega := (0,1)^2$ and $z_0 = 0.25$, $x_1^0 = x_2^0 = 0.5$. 
We choose the parameter $\varepsilon = 10^{-6}$, $\boldsymbol{b} = (2,3)^\top$ and $\alpha = 
1.0$. 
For the solution \eqref{BruchhaeuserEq:movghump} the right-hand side function $f$ is 
calculated from the partial differential equation. Boundary conditions 
are given by the exact solution. Our target quantity is chosen as 
\begin{equation}
\label{BruchhaeuserEq:HumpTarget}
J(u) = \frac{1}{\Vert e\Vert_{L^2(\Omega)}} (e,u)_{\Omega}\,.
\end{equation}
In Fig.~\ref{BruchhaeuserFig:MH_L2dofs} we compare the convergence behavior of the 
proposed DWR approach with a global mesh refinement strategy. The corresponding 
solution profiles are visualized in Fig.~\ref{BruchhaeuserFig:MHsol}. The adaptively 
generated mesh is presented in Fig.~\ref{BruchhaeuserFig:MeshMH}. The DWR based adaptive 
mesh adaptation is clearly superior to the global refinement in terms of accuray over 
degrees of freedom. While the globally refined solution is still 
perturbed by undesired oscillations within the circular layer and behind 
the hump in the direction of the flow field $\boldsymbol{b}$, the adaptively computed solution 
exhibits an almost perfect solution profile for even less degrees of freedom. 

\begin{figure}
\centering
\subfloat[$L^2$-error over degrees of freedom for global and DWR adaptive 
mesh refinement.]{
  \centering
\begin{minipage}{.45\linewidth}
\centering
\begin{tikzpicture}
\begin{axis}[%
width=2.0in,
height=1.1in,
scale only axis,
/pgf/number format/.cd, 1000 sep={},
xmode=log,
ymode=log,
yminorticks=true,
legend style={legend pos=south west},
legend style={draw=black,fill=white,legend cell align=left},
legend entries = {
  {\textbf{gl} $Q_1$},
  {\textbf{ad} $Q_1/Q_2$},
  {\textbf{ad} $Q_1/Q_3$},
  {\textbf{ad} $Q_1/Q_4$},
}
]

\addplot [
color=black,
solid,
line width=1.0pt,
mark=+,
mark size = 1.5,
mark options={solid,black}
]
table[row sep=crcr]{
81      2.86363e-01 \\ 
289     2.59459e-01 \\ 
1089    9.83539e-02 \\ 
4225    5.65332e-02 \\ 
16641   2.28166e-02 \\ 
66049   1.05549e-02 \\ 
263169  4.63511e-03 \\ 
1050625 1.92367e-03 \\
};

\addplot [
color=black,
dashed,
line width=1.0pt,
mark=square*,
mark size = 1.5,
mark options={solid,black}
]
table[row sep=crcr]{
81     2.86363e-01 \\ 
206    2.52635e-01 \\ 
608    1.04744e-01 \\ 
1516   5.86416e-02 \\ 
3623   2.78384e-02 \\ 
7290   1.41160e-02 \\ 
13902  6.83149e-03 \\ 
25232  3.14782e-03 \\ 
49691  1.31823e-03 \\ 
111871 4.82777e-04 \\ 
295106 1.63420e-04 \\ 
842390 6.26958e-05 \\ 
};
\end{axis}
\end{tikzpicture}
\end{minipage}
\label{BruchhaeuserFig:MH_L2dofs}
}
\hfill
\subfloat[Adaptive mesh for target quantity \eqref{BruchhaeuserEq:HumpTarget} with 56222 
degrees of freedom.]{
\centering
\begin{minipage}{.45\linewidth}
\centering
\includegraphics[width=3.3cm]{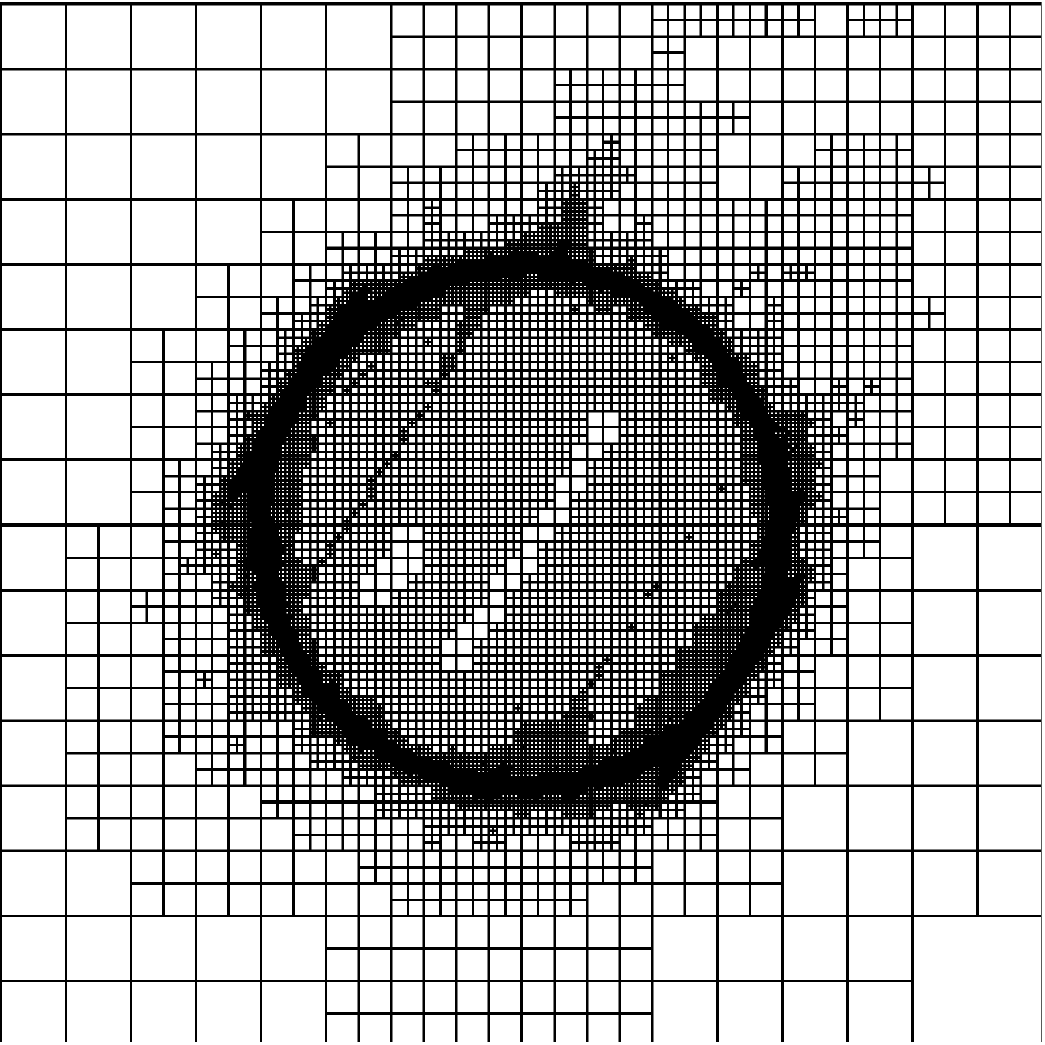}
\end{minipage}
\label{BruchhaeuserFig:MeshMH}
}
\caption{Comparison of $L^2$-errors and visualization of the adaptive mesh 
for Example \ref{BruchhaeuserExample1}.}
\label{BruchhaeuserFig:MHL2andMesh}
\end{figure}
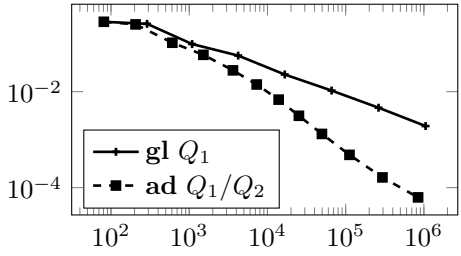
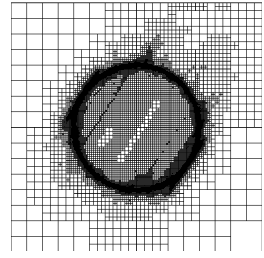

\begin{figure}
\centering
\subfloat[Global Refinement.]{
  \centering
\begin{minipage}{.45\linewidth}
\centering
\includegraphics[width=5.8cm]{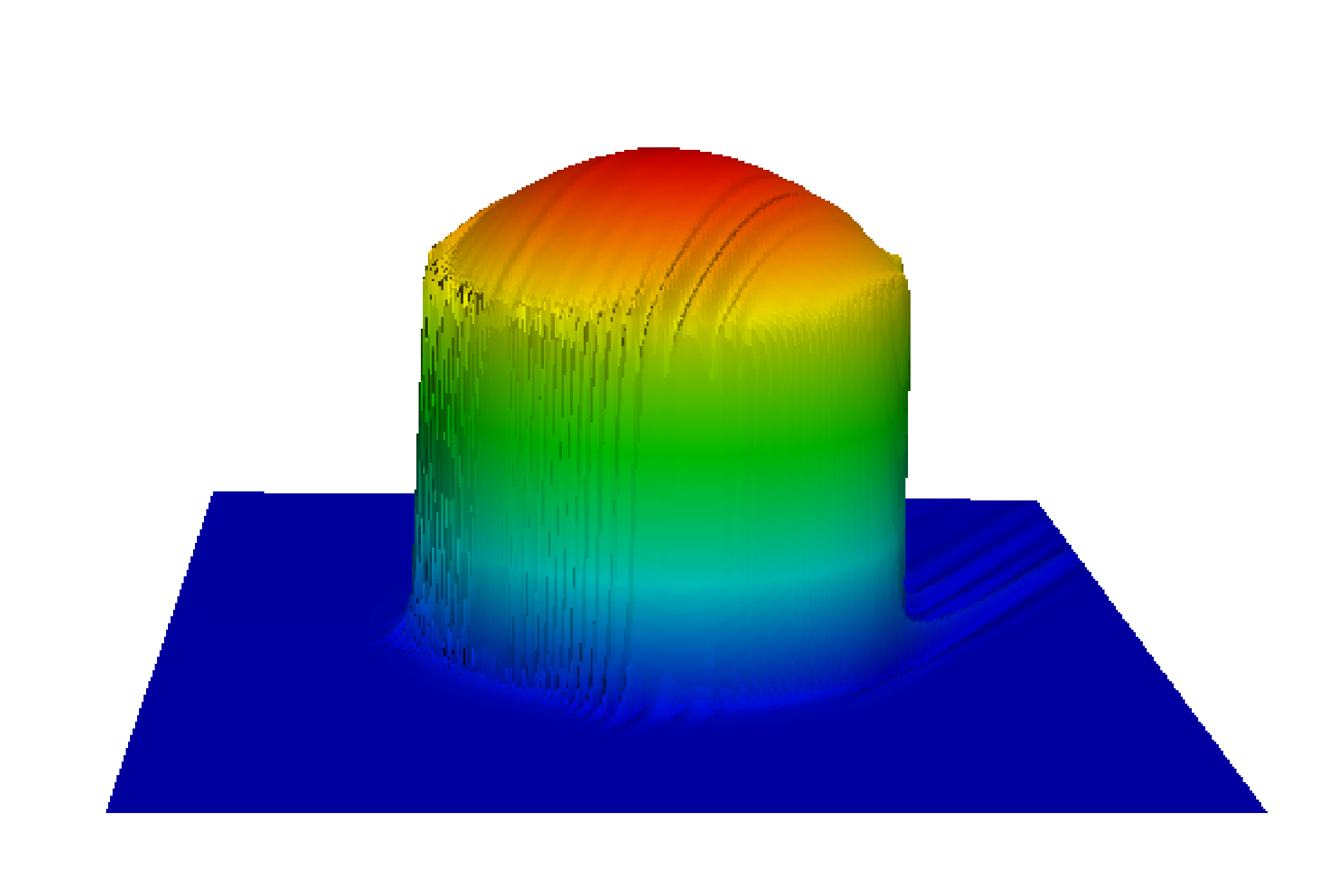}
\end{minipage}
\label{BruchhaeuserFig:GR}
}
\hfill
\subfloat[Adaptive Refinement.]{
  \centering
\begin{minipage}{.45\linewidth}
\centering
\includegraphics[width=5.8cm]{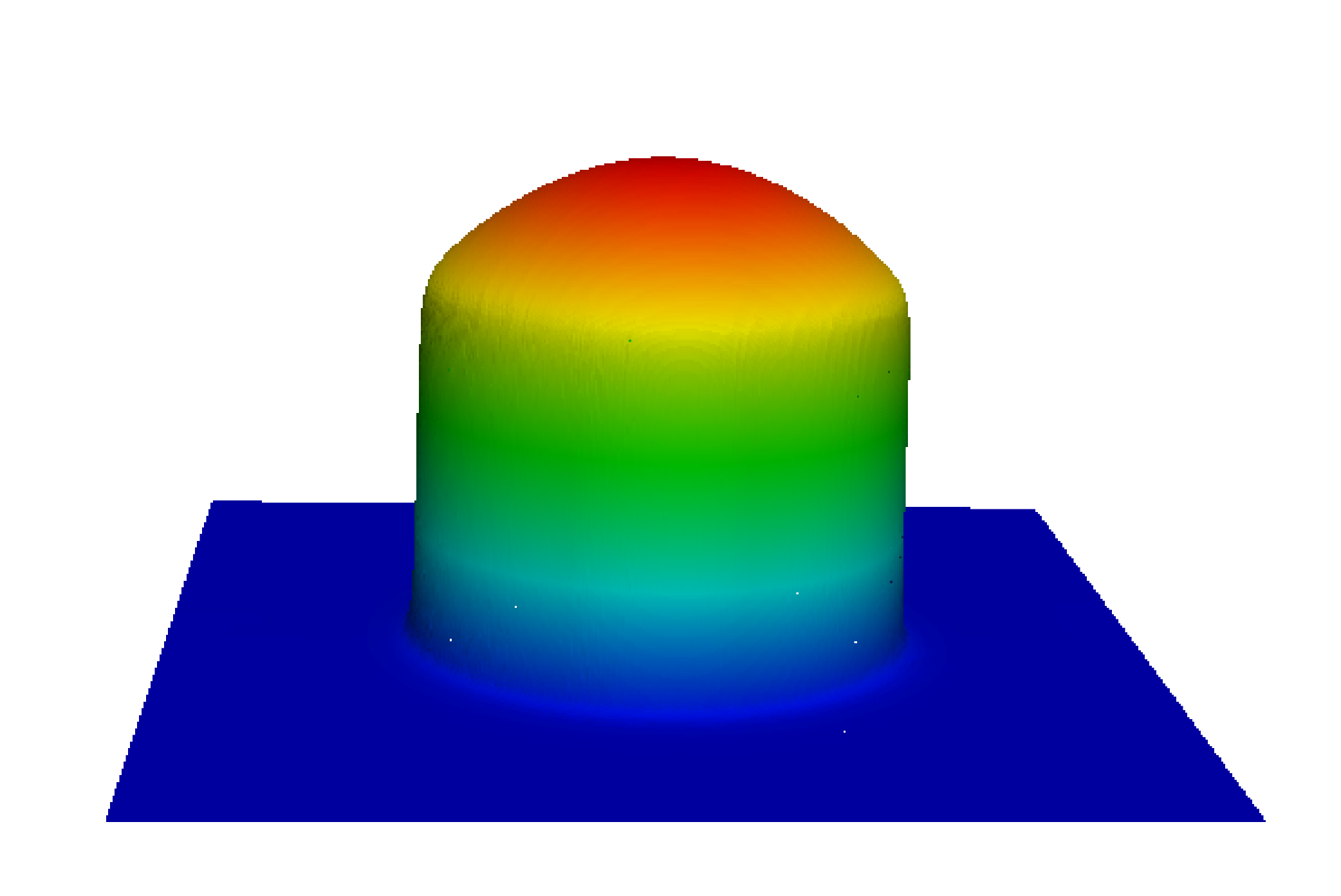}
\end{minipage}
\label{BruchhaeuserFig:AR}
}
\caption{Stabilized solution profile on a globally refined mesh with 66049 
degrees of freedom (\ref{BruchhaeuserFig:GR}) and on an adaptively refined mesh with error control 
by the target quantity \eqref{BruchhaeuserEq:HumpTarget} with 56222 degrees of 
freedom (\ref{BruchhaeuserFig:AR}) for Example \ref{BruchhaeuserExample1}.}
\label{BruchhaeuserFig:MHsol}
\end{figure}


In Fig.~\ref{BruchhaeuserFig:IeffMH} we present the calculated effectivity indices 
\eqref{BruchhaeuserEq:Ieff} for solving the primal and dual problem in different 
pairs of finite element spaces based on the $Q_k$ element of polynomials of maximum 
degree $k$ in each variable. Considering the 
$Q_1$ based approximation of the primal problem, we note that by increasing the 
polynomial degree of the dual solution from $Q_2$ to $Q_4$ the mesh 
adaptation process reaches the stopping criterion faster and requires less degrees of 
freedom. This observation is reasonable, since a higher order approximation of the 
dual problem is closer to its exact solution of the dual problem, which is part of 
the error representation \eqref{BruchhaeuserEq:localER}. Thus we conclude that a 
better approximation of the weights provides a higher accuracy of the error estimator.
This observation is also confirmed by the pairs of $Q_2/Q_3$ with $Q_2/Q_4$ based finite 
element spaces. Nevertheless, the difference for using higher-order finite elements for 
solving the dual problem is not that significant, even less if we take into
account the higher computational costs for solving the algebraic form of the dual problem 
for an increasing order of the piecewise polynomials. Using pairs of $Q_k/Q_{k+1}$ based 
elements for the approximation of the primal and dual problem, the error 
estimator gets worse for increasing values of the parameter $k$. This observation is in 
good agreement with the results in \cite[Example 3]{BruchhaeuserEW17}. A reason for 
this behavior is given by the observation that for increasing values of $k$ the mesh is 
less refined for the same number of degrees of freedom. Therefore less cells are available 
to capture the strong gradients of the exact solution. This argues for choosing smaller 
values of $k$ in the application of our DWR based approach.


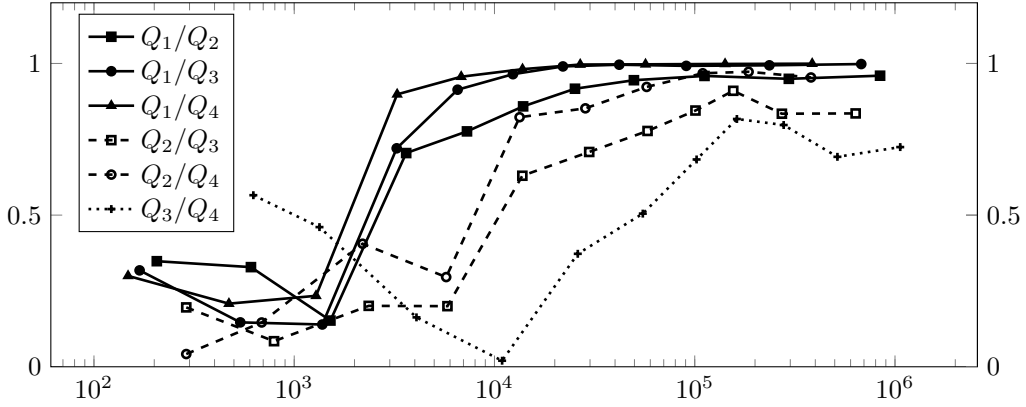
\begin{figure}
\centering
\begin{minipage}{.9\linewidth}
\centering
\begin{tikzpicture}
\begin{axis}[%
width=4.8in,
height=1.9in,
scale only axis,
/pgf/number format/.cd, 1000 sep={},
xmode=log,
ymin=0.,
ymax=1.2,
yminorticks=true,
legend style={legend pos=north west},
legend style={draw=black,fill=white,legend cell align=left},
legend entries = {
  {$Q_1/Q_2$},
  {$Q_1/Q_3$},
  {$Q_1/Q_4$},
  {$Q_2/Q_3$},
  {$Q_2/Q_4$},
  {$Q_3/Q_4$},
}
]

\addplot [
color=black,
solid,
line width=1.0pt,
mark=square*,
mark size = 1.5,
mark options={solid,black}
]
table[row sep=crcr]{
206    0.34805\\ 
608    0.32849\\ 
1516   0.15237\\ 
3623   0.70455\\ 
7290   0.77605\\ 
13902  0.85846\\ 
25232  0.91668\\ 
49691  0.94474\\ 
111871 0.95881\\ 
295106 0.94930\\ 
842390 0.95961\\
};
\addplot [
color=black,
solid,
line width=1.0pt,
mark=*,
mark size = 1.5,
mark options={solid,black}
]
table[row sep=crcr]{
169    0.31781\\ 
538    0.14636\\ 
1381   0.13935\\ 
3249   0.72024\\ 
6545   0.91346\\ 
12366  0.96456\\ 
21968  0.99017\\ 
41974  0.99585\\ 
90709  0.99201\\ 
235885 0.99384\\ 
678368 0.99800\\
};
\addplot [
color=black,
solid,
line width=1.0pt,
mark=triangle*,
mark size = 1.5,
mark options={solid,black}
]
table[row sep=crcr]{
148       0.299602\\ 
472       0.208077\\ 
1285      0.233949\\ 
3266      0.898114\\ 
6821      0.956222\\ 
13806     0.981114\\ 
26801     0.996459\\ 
56779     0.997107\\ 
141736    0.998407\\ 
385973    0.999241\\ 
};
\addplot [
color=black,
dashed,
line width=1.0pt,
mark=square,
mark size = 1.5,
mark options={solid,black}
]
table[row sep=crcr]{
289       0.19575\\ 
793       0.08433\\ 
2355      0.20084\\ 
5823      0.19980\\ 
13774     0.62980\\ 
29720     0.70814\\ 
58124     0.77744\\ 
100832    0.84480\\ 
155696    0.90983\\ 
272336    0.83430\\ 
636167    0.83525\\ 
};
\addplot [
color=black,
dashed,
line width=1.0pt,
mark=o,
mark size = 1.5,
mark options={solid,black}
]
table[row sep=crcr]{
289     0.04194\\ 
688     0.14602\\ 
2194    0.40566\\ 
5739    0.29567\\ 
13346   0.82268\\ 
28438   0.85201\\ 
57514   0.92255\\ 
109637  0.96761\\ 
185745  0.97255\\ 
380663  0.95379\\ 
};
\addplot [
color=black,
dotted,
line width=1.0pt,
mark=+,
mark size = 1.5,
mark options={solid,black}
]
table[row sep=crcr]{
625     0.56560 \\
1337    0.46044 \\
4078    0.16289 \\
10918   0.01999 \\
26061   0.37244 \\
55042   0.50493 \\
102162  0.68346 \\
162365  0.81690 \\
278043  0.79704 \\
515776  0.69216 \\
1061925 0.72426 \\
};
\end{axis}
\begin{axis}[
width=4.8in,
height=1.9in,
scale only axis,
/pgf/number format/.cd, 1000 sep={},
ymin=0,
ymax=1.2,
hide x axis,
axis y line*=right,
]
\addplot[opacity=0]{1.2};
\end{axis}
\end{tikzpicture}
\end{minipage}
\caption{Effectivity indices over degrees of freedom for the target quantity 
\eqref{BruchhaeuserEq:HumpTarget} and different polynomial degrees for 
Example \ref{BruchhaeuserExample1}.}
\label{BruchhaeuserFig:IeffMH}
\end{figure}


\subsection{Example 2 (Point-value error control, 2d).}
\label{BruchhaeuserExample2}

In this example we study the application of our approach for different target 
functionals and a sequence of decreasing diffusion coefficients. 
Thereby we aim to analyze the robustness of the approach with respect to the 
small perturbation parameter $\varepsilon$ in \eqref{BruchhaeuserEq:cdr}. 
If convection-dominated problems are considered, or also often for applications of 
practical interest, local quantities are of greater interest than global ones. The DWR 
approach offers the appreciable advantage over standard a posteriori error estimators 
that an error control in an arbitrary user-chosen quantity and not 
only in the global $L^2$ norm, as in Example 1, or a norm of energy type can be 
obtained. Since the error representation is exact up to higher order terms (cf.\ Thm.\ 
\ref{BruchhaeuserThm:LocPresent}), robustness with respect to the perturbation parameter 
$\varepsilon$ can be expected to be feasible. Of course, the approximation of 
the dual solution $z$ in \eqref{BruchhaeuserEq:localER} adds a source of 
uncertainty in the error representation. In the sequel, we evaluate the potential of our 
approach with respect to these topics for different target functionals. As a benchmark 
problem we consider problem \eqref{BruchhaeuserEq:cdr} for the solution (cf.\ 
\cite[Example 4.2]{BruchhaeuserLR06})
\begin{equation}
\label{BruchhaeuserEq:SolExp2}
u(\boldsymbol{x}) = \frac{1}{2}\left(1-\operatorname{tanh}\frac{2x_1 - 
x_2-0.25}{\sqrt{5\varepsilon}}\right)
\end{equation}
with corresponding right-hand side function $f$. Further, $\Omega = (0,1)^2$, $\alpha = 
1.0$, $\boldsymbol{b} = \frac{1}{\sqrt{5}}(1,2)^\top$. The Dirichlet boundary condition is given 
by the exact solution. The solution is characterized by an 
interior layer of thickness 
$\mathcal O(\sqrt{\varepsilon}\vert \operatorname{ln} \varepsilon\vert)$. 
We study the target functionals 
\begin{equation*}
J_{L^2}(u) = \frac{1}{\Vert e\Vert_{L^2(\Omega)}} 
(e,u)_{\Omega}\,, \qquad J_1(u) = \int_\Omega u\,\mathrm{d}\boldsymbol{x} \qquad 
\mbox{and} \qquad J_2(u) = u(\boldsymbol{x}_e)\,,
\end{equation*}
where $e:=u-u_h$ and with a user-prescribed control point $\boldsymbol{x}_e 
= \left(\frac{5}{16},\frac{3}{8}\right)$ that is located in the interior of the 
layer. In our computations we regularize $J_2(\cdot)$ by 
\begin{equation*}
J_{r}(u) = \frac{1}{\vert B_r\vert}\int_{B_r} u(\boldsymbol{x})\,\mathrm{d}\boldsymbol{x}\,,
\end{equation*}
where the ball $B_r$ is defined by $B_r = \left\{\boldsymbol{x}\in\Omega \mid \|
\boldsymbol{x}-\boldsymbol{x}_e\| < r \right\}$ with small radius $r>0$.

Here, all test cases are solved by using the $Q_1/Q_2$ pair of finite elements for 
the primal and dual problem which is due to the observations depicted in Example~1. In 
Table~\ref{BruchhaeuserTab:TanhIeff} and Fig.~\ref{BruchhaeuserFig:IeffTH6-8}
we present the effectivity indices of the proposed DWR approach applied to the 
stabilized approximation scheme \eqref{BruchhaeuserEq:ASprimal} for a sequence 
of vanishing diffusion coefficients. 
For the target functionals $J_{L^2}(\cdot)$ and $J_1(\cdot)$ the effectivity 
indices nicely converge to one for an increasing number of degrees of freedom. Moreover, 
the expected convergence behavior is robust with respect to the small diffusion 
parameter $\varepsilon$. For the more challenging error control of a point-value, which 
however can be expected to be of higher interest in practice, the effectivity 
indices also convergences nicely to one. This is in good agreement with effectivity 
indices for point-value error control that are given in other 
works of the literature for the pure Poisson problem; cf.~\cite[p.~45]{BruchhaeuserBR03}. 
We note that in the case of a point-value error control the target functional lacks the 
regularity of the right-hand side term in the dual problem that is typically needed to 
ensure the existence and regularity of weak solutions; cf.~\cite[Chapter 
6.2]{BruchhaeuserE10}. 
However, no impact of this lack of regularity is observed in the computational 
studies. Thus, for all target functionals a robust convergence behavior is 
ensured for this test case.

\begin{table} 
\centering

\begin{minipage}{\linewidth}
\centering

\resizebox{1.\linewidth}{!}{%
\begin{tabular}{cc|cc|cc|cc|cc|cc|cc|cc|cc}
\toprule
\toprule
\multicolumn{6}{c|}{$\varepsilon=10^{-6}$} &
\multicolumn{6}{c|}{$\varepsilon=10^{-7}$} &
\multicolumn{6}{c}{$\varepsilon=10^{-8}$}\\
\midrule
\multicolumn{2}{c|}{$J_{L^2}$} & \multicolumn{2}{c|}{$J_1$} & \multicolumn{2}{c|}{$J_r$} &
\multicolumn{2}{c|}{$J_{L^2}$} & \multicolumn{2}{c|}{$J_1$} & \multicolumn{2}{c|}{$J_r$} &
\multicolumn{2}{c|}{$J_{L^2}$} & \multicolumn{2}{c|}{$J_1$} & \multicolumn{2}{c}{$J_r$}\\
\midrule
dofs & $\mathcal I_{\mathrm{eff}}$ &
  dofs & $\mathcal I_{\mathrm eff}$ &
  dofs & $\mathcal I_{\mathrm eff}$ &
dofs & $\mathcal I_{\mathrm{eff}}$ &
  dofs & $\mathcal I_{\mathrm eff}$ &
  dofs & $\mathcal I_{\mathrm eff}$ &
dofs & $\mathcal I_{\mathrm{eff}}$ &
  dofs & $\mathcal I_{\mathrm eff}$ &
  dofs & $\mathcal I_{\mathrm eff}$ \\
\midrule
\midrule
 4531 & 0.72 & 5355  & 0.14  & 4334  & 1.09 & 1643   & 0.32 & 4034    & 0.12 
 & 753     & 0.73 & 1477 & 0.46 & 2810 & 0.56 & 1841 & 1.00 \\
 7603 & 0.84 & 9458  & 1.92  & 6468  & 1.16 & 2772   & 0.37 & 8383    & 6.38 
 & 1841    & 0.99 & 5668 & 0.54 & 10925 & 0.90 & 3301 & 1.22 \\
 13319 & 0.92 & 16067  & 1.73 & 13851  & 1.58 & 5094   & 0.43 & 16314   & 0.01 
 & 6358    & 1.34 & 10005 & 0.56 & 22028 & 0.05 & 6411 & 1.46 \\
24418 & 0.98 & 28584 & 0.85  & 19670  & 1.19 & 10086  & 0.53 & 28310   & 0.13 
& 12674   & 1.23 & 17974 & 0.59 & 41088 & 1.72 & 12904 & 1.72 \\
42174 & 1.00 & 52024 & 3.06  & 38002  & 0.99 & 21071  & 0.67 & 44111   & 0.61 
& 25558   & 1.09 & 41402 & 0.66 & 71646 & 0.26 & 27039 & 1.74 \\
76341 & 1.01 & 95006 & 1.07  & 51603 & 1.20 & 46172  & 0.84 & 72705   & 1.07 
& 54549   & 0.96 & 90486 & 0.78 & 141031 & 1.96 & 58254 & 1.46 \\
126757 & 0.99 & 179893 & 1.01  & 119171 & 0.94 & 103077 & 0.97 & 178757  & 0.85 
& 139531  & 1.31 & 253203 & 0.90 & 305855 & 1.01 & 123436 & 1.24 \\
224160 & 1.01 & 307864 & 1.00  & 357046 & 0.90 & 240672 & 1.02 & 433232  & 1.00 
& 387749  & 1.27 & 502287 & 1.00 & 608497 & 1.04 & 274188 & 1.13 \\
409008 & 0.99 & 560046 & 1.00  & 571577 & 1.14 & 580812 & 1.00 & 1003495 & 1.00 
& 1181627 & 1.01 & 801381 & 1.01 & 856320 & 1.01 & 691860 & 1.08 \\
\bottomrule
\end{tabular}}
\end{minipage}
\caption{Effectivity indices for the target quantities $J_{L^2}$\, $J_1$ and 
$J_r$ and different decreasing perturbation parameters for Example \ref{BruchhaeuserExample2}.}
\label{BruchhaeuserTab:TanhIeff}
\end{table}

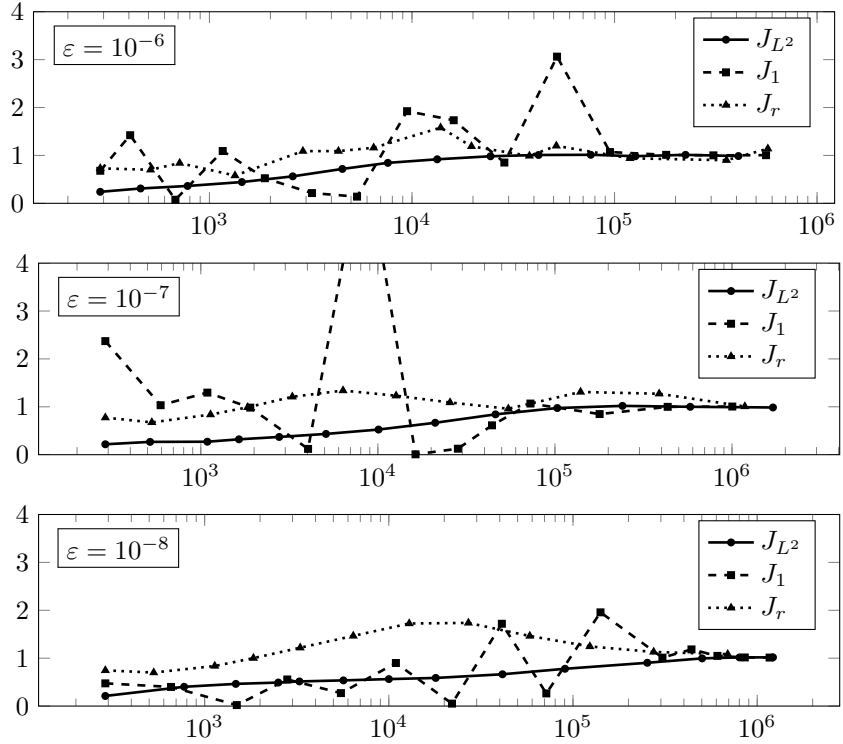
\begin{figure}[h]
\centering
\begin{minipage}{.8\linewidth}
\centering
\begin{tikzpicture}
\begin{axis}[%
width=4.15in,
height=1.0in,
scale only axis,
/pgf/number format/.cd, 1000 sep={},
xmode=log,
ymin=0,
ymax=4.,
yminorticks=true,
legend style={legend pos=north east},
legend style={draw=black,fill=white,legend cell align=left},
legend entries = {
  {$J_{L^2}$},
  {$J_1$},
  {$J_r$},
}
]
\node at (axis cs:340, 3.3) {\fbox{$\varepsilon=10^{-6}$}};
\addplot [
color=black,
solid,
line width=1.0pt,
mark=*,
mark size = 1.,
mark options={solid,black}
]
table[row sep=crcr]{
289    0.24\\
457    0.30981\\ 
780    0.36221\\ 
1446   0.44413\\ 
2578   0.56326\\ 
4531   0.71800\\ 
7603   0.84463\\ 
13319  0.91903\\ 
24418  0.98263\\ 
42174  1.00922\\ 
76341  1.01295\\ 
126757 0.98429\\ 
224160 1.01183\\ 
409008 0.98693\\ 
};
\addplot [
color=black,
dashed,
line width=1.0pt,
mark=square*,
mark size = 1.,
mark options={solid,black}
]
table[row sep=crcr]{
289    0.67938\\ 
406    1.42246\\ 
679    0.07514\\ 
1163   1.09174\\ 
1878   0.52360\\ 
3208   0.21393\\ 
5355   0.13955\\ 
9458   1.92381\\ 
16067  1.73461\\ 
28584  0.85146\\ 
52024  3.06259\\ 
95024  1.07054\\ 
179893 1.01165\\ 
307864 1.00189\\ 
560046 1.00476\\ 
};
\addplot [
color=black,
dotted,
line width=1.0pt,
mark=triangle*,
mark size = 1.,
mark options={solid,black}
]
table[row sep=crcr]{
289  0.73\\ 
512  0.70\\
713  0.84\\
1337 0.58\\
2894 1.09\\
4334 1.09\\
6468 1.16\\
13851 1.58\\
19670 1.19\\
38002 0.99\\
51603 1.20\\
119171  0.94\\ 
357046  0.90\\ 
571577  1.14\\ 
};
\end{axis}
\end{tikzpicture}
\end{minipage}
\label{BruchhaeuserFig:Ieff6}

\begin{minipage}{.8\linewidth}
\centering
\begin{tikzpicture}
\begin{axis}[%
width=4.15in,
height=1.0in,
scale only axis,
/pgf/number format/.cd, 1000 sep={},
xmode=log,
ymin=0.,
ymax=4,
yminorticks=true,
legend style={legend pos=north east},
legend style={draw=black,fill=white,legend cell align=left},
legend entries = {
  {$J_{L^2}$},
  {$J_1$},
  {$J_r$},
}
]
\node at (axis cs:345, 3.3) {\fbox{$\varepsilon=10^{-7}$}};
\addplot [
color=black,
solid,
line width=1.0pt,
mark=*,
mark size = 1.,
mark options={solid,black}
]
table[row sep=crcr]{
289     0.217493\\
517     0.267445\\
1089    0.27084\\ 
1643    0.32119\\ 
2772    0.37044\\ 
5094    0.43361\\ 
10086   0.52509\\ 
21071   0.66639\\ 
46172   0.84212\\ 
103077  0.97433\\ 
240672  1.01933\\ 
580812  1.00006\\ 
1702384 0.98586\\ 
};

\addplot [
color=black,
dashed,
line width=1.0pt,
mark=square*,
mark size = 1.,
mark options={solid,black}
]
table[row sep=crcr]{
289     2.37167\\
593     1.03321\\
1089    1.29625\\ 
1902    0.98299\\ 
4034    0.12107\\ 
8383    6.37956\\ 
16314   0.00773\\ 
28310   0.12527\\ 
44111   0.61316\\ 
72705   1.06796\\ 
178757  0.84844\\ 
433232  0.99896\\ 
1003495 1.00575\\ 
};

\addplot [
color=black,
dotted,
line width=1.0pt,
mark=triangle*,
mark size = 1.,
mark options={solid,black}
]
table[row sep=crcr]{
289     0.77417\\ 
529     0.67431\\ 
1135    0.83784\\ 
1841    0.98937\\ 
3289    1.20918\\ 
6358    1.33957\\ 
12674   1.23139\\ 
25558   1.08985\\ 
54549   0.95962\\
139531  1.30964 \\
387749  1.27269 \\
1181627 1.00678 \\
};
\end{axis}
\end{tikzpicture}
\end{minipage}
\label{BruchhaeuserFig:Ieff7}

\begin{minipage}{.8\linewidth}
\centering
\begin{tikzpicture}
\begin{axis}[%
width=4.15in,
height=1.0in,
scale only axis,
/pgf/number format/.cd, 1000 sep={},
xmode=log,
ymin=0.,
ymax=4.,
yminorticks=true,
legend style={legend pos=north east},
legend style={draw=black,fill=white,legend cell align=left},
legend entries = {
  {$J_{L^2}$},
  {$J_1$},
  {$J_r$},
}
]
\node at (axis cs:340, 3.3) {\fbox{$\varepsilon=10^{-8}$}};
\addplot [
color=black,
solid,
line width=1.0pt,
mark=*,
mark size = 1.,
mark options={solid,black}
]
table[row sep=crcr]{
289  0.21117\\ 
774  0.40254\\ 
1477  0.46318\\ 
2512  0.49118\\ 
3271  0.51445\\ 
5668  0.53637\\ 
10005  0.56328\\ 
17974  0.58717\\ 
41402  0.66404\\ 
90486  0.77843\\ 
253203  0.90362\\ 
502287  0.99594\\ 
801387  1.01605\\ 
1217439  1.01700\\ 
};

\addplot [
color=black,
dashed,
line width=1.0pt,
mark=square*,
mark size = 1.,
mark options={solid,black}
]
table[row sep=crcr]{
289     0.47430\\ 
657     0.39940\\ 
1495    0.02082\\ 
2810    0.55544\\ 
5495    0.27093\\ 
10925   0.89903\\ 
22028   0.05145\\ 
41088   1.71868\\ 
71646   0.26428\\ 
141031  1.95700\\ 
305855  1.01163\\
439715  1.18297\\ 
608497  1.04809\\ 
856320  1.01507\\
1167807 1.01292\\
};

\addplot [
color=black,
dotted,
line width=1.0pt,
mark=triangle*,
mark size = 1.,
mark options={solid,black}
]
table[row sep=crcr]{
289    0.74524\\
529    0.69979\\ 
1135   0.83682\\ 
1841   1.00114\\ 
3301   1.21824\\ 
6411   1.46463\\ 
12904  1.72397\\ 
27039  1.73650\\ 
58254  1.46300\\ 
123436 1.24481\\
274188 1.12988\\
691860 1.08252\\
};
\end{axis}
\end{tikzpicture}
\end{minipage}
\label{BruchhaeuserFig:Ieff8}

\caption{Effectivity indices over degrees of freedom for different target 
functionals and decreasing diffusion coefficient for 
Example \ref{BruchhaeuserExample2}.}
\label{BruchhaeuserFig:IeffTH6-8}
\end{figure}


For completeness, in Figure \ref{BruchhaeuserFig:TanhSol} we visualize the 
computed solution profiles and adaptive meshes for an error control based on the 
local target functional $J_r(\cdot)$ and the global target functional 
$J_{L^2}(\cdot)$, respectively. This test case nicely illustrates the potential of the 
DWR approach. For the point-value error control the refined mesh cells are located close 
to the specified point of interest and along those cells that affect the 
point-value error by means of transport in the direction of the flow field~$\boldsymbol{b}$. 
Furthermore, the mesh cells without strong impact on the solution close to the control 
point are coarsened further. Even though a rough approximation of the sharp interface is 
obtained in downstream direction from the viewpoint of the control point, in its 
neighborhood an excellent approximation of the sharp layer is ensured by the approach. A 
highly economical mesh along with a high quality in the computation of the user-specified 
goal quantity is thus obtained. In contrast to this, the global error control of 
$J_{L^2}(\cdot)$ provides a good approximation of the solution in the whole domain by 
adjusting the mesh along the complete layer.

\begin{figure}
\centering
\subfloat[Point-value error control.]{
  \centering
\begin{minipage}{.45\linewidth}
\centering
\includegraphics[width=4.2cm]{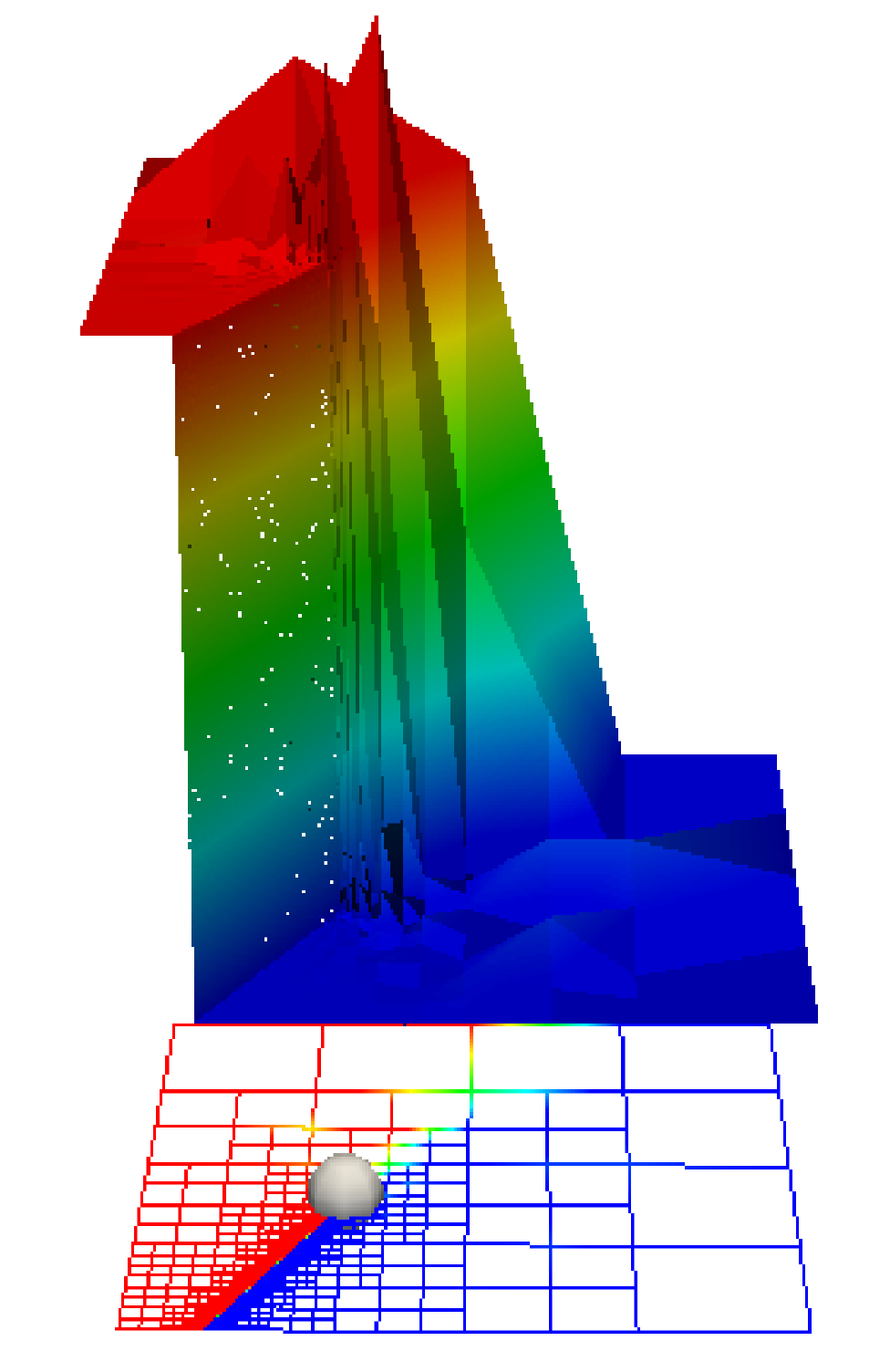}
\end{minipage}
\label{BruchhaeuserFig:PVEC}
}
\hfill
\subfloat[Global error control.]{
  \centering
\begin{minipage}{.45\linewidth}
\centering
\includegraphics[width=4.5cm]{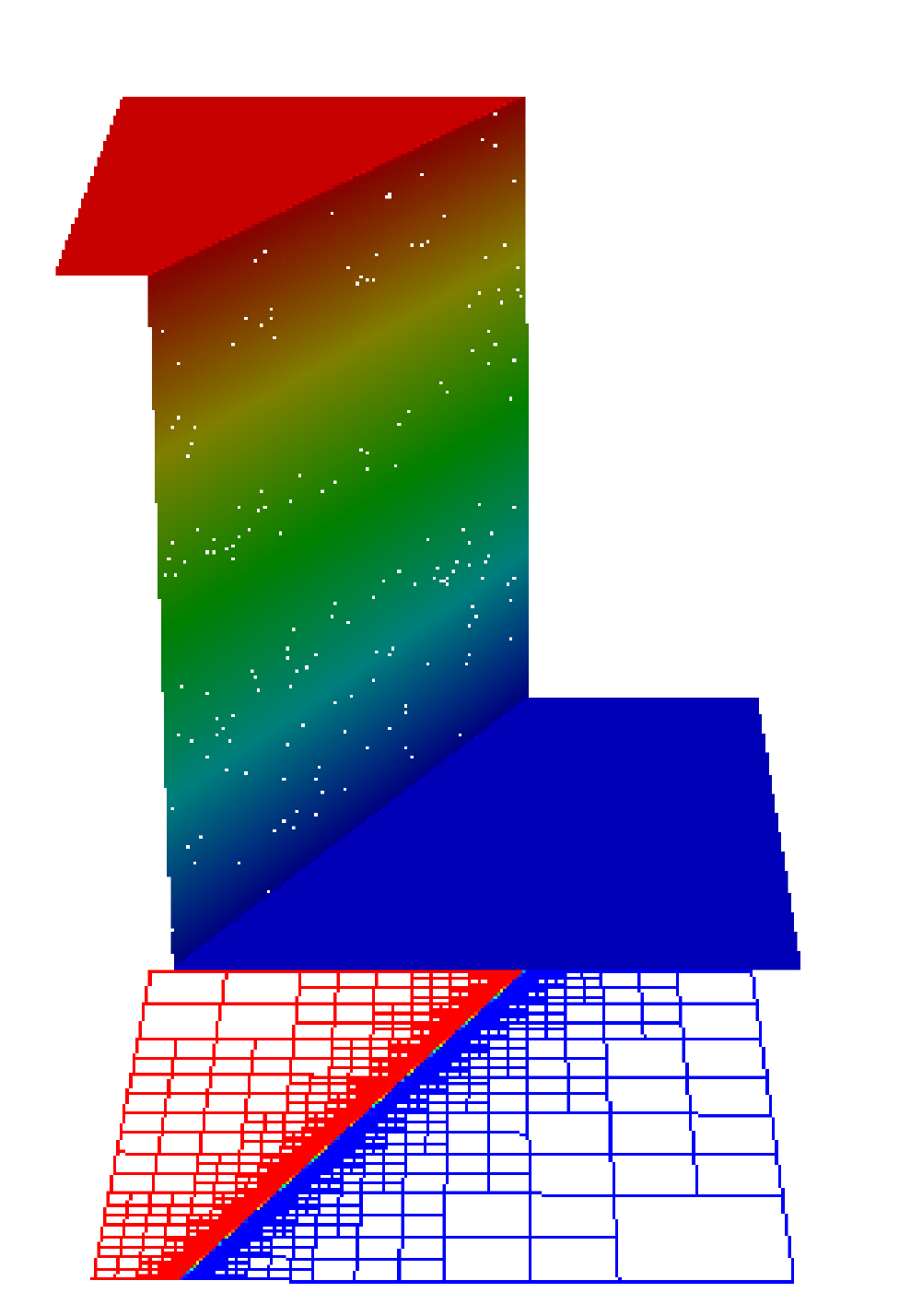}
\end{minipage}
\label{BruchhaeuserFig:GEC}
}
\caption{Point-value error control by $J_2$ (\ref{BruchhaeuserFig:PVEC}) and global error control
by $J_{L^2}$
(\ref{BruchhaeuserFig:GEC}) by the DWR approach for Example \ref{BruchhaeuserExample2}.}
\label{BruchhaeuserFig:TanhSol}
\end{figure}

\subsection{Example 3 (Variable convection field, 3D).} 
\label{BruchhaeuserExample3}
In our last test case we apply the approach to a three-dimensional test case 
which represents a more challenging task. 
Moreover, we consider a velocity field $\boldsymbol{b}$ depending on the space variable 
$\boldsymbol{x}$. Precisely, we consider problem \eqref{BruchhaeuserEq:cdr} with the unit 
cube $\Omega=(0,1)^3$, $\varepsilon = 10^{-6}$, $\alpha=1$, 
$\boldsymbol{b}=(-x_2,x_1,0)^\top$ and $f\equiv 0$. The boundary conditions are given by 
$\frac{\partial u}{\partial \boldsymbol{n}}=0$ on $\Gamma_N=\{\boldsymbol{x}~\in~\Omega \mid x_1 = 
0\}$, $u = 1$ on $\Gamma_{D_1}=\{\boldsymbol{x}\in\Omega \mid 0.4 \leq x_1 \leq 0.6,
x_2 = 0, 0.4 \leq x_3 \leq 0.6 \}$, and $u = 0$ on $\Gamma_{D_2}=
\partial\Omega\backslash\{\Gamma_N \cup \Gamma_{D_1}\}$. Thus, by the boundary part 
$\Gamma_{D_1}$ we model an inflow region (area) where the transport quantity modelled by 
the unknown $u$ is injected; cf.\ Fig.~\ref{BruchhaeuserFig:VariableConField}. $\Gamma_N$ 
models an outflow boundary. Prescribing a homogeneous Dirichlet condition on 
$\Gamma_{D_2}$ is done for the sake of simplicity and of no real relevance for the test 
setting. The target functional aims at the control of the solution's mean value in a 
smaller, inner domain $\Omega_{In}=[0,0.1]\times[0.4,0.6]\times[0.4,0.6]$, and is given by
$$J_4(u)=\int_{\Omega_{In}}u \,\mathrm{d}\boldsymbol{x}\,.$$
In the context of applications, the transport quantity $u$ is thus measured and 
controlled in the small region of interest $\Omega_{In}$.

Figure \ref{BruchhaeuserFig:VariableConField} illustrates the computed adaptively 
generated meshes for some of the DWR iteration steps. For visualization purposes, two 
surfaces with corresponding mesh distribution are shown for each grid, the bottom surface 
and the surface in the domain's center with respect to the $x_3$ direction. We note that 
the postprocessed solutions are visualized on a grid for the respective surfaces. The 
cells on the surfaces are triangular-shaped since the used visualization software 
\texttt{ParaView} is based on triangular-shaped elements. Similar to the previous 
test case of a point-value error control, the refinement is located on those cells that 
affect the mean value error control. Here, the cells close to the two inner layers 
aligned in the flow direction $\boldsymbol{b}$ are strongly refined. This refinement 
process is obvious since the inner and control domain $\Omega_{In}$ is chosen to have 
exactly the same dimensions as the channel-like extension of the boundary segment 
$\Gamma_{D_1}$ along the flow direction into the domain $\Omega$. Outside the inner domain 
$\Omega_{In}$ and the channel-like domain of transport the mesh cells are coarsened for  
an increasing number of DWR iteration steps.

\begin{figure}[h]
\centering
\subfloat[]{
  \centering
\begin{minipage}{.45\linewidth}
\centering
\includegraphics[width=5.5cm]{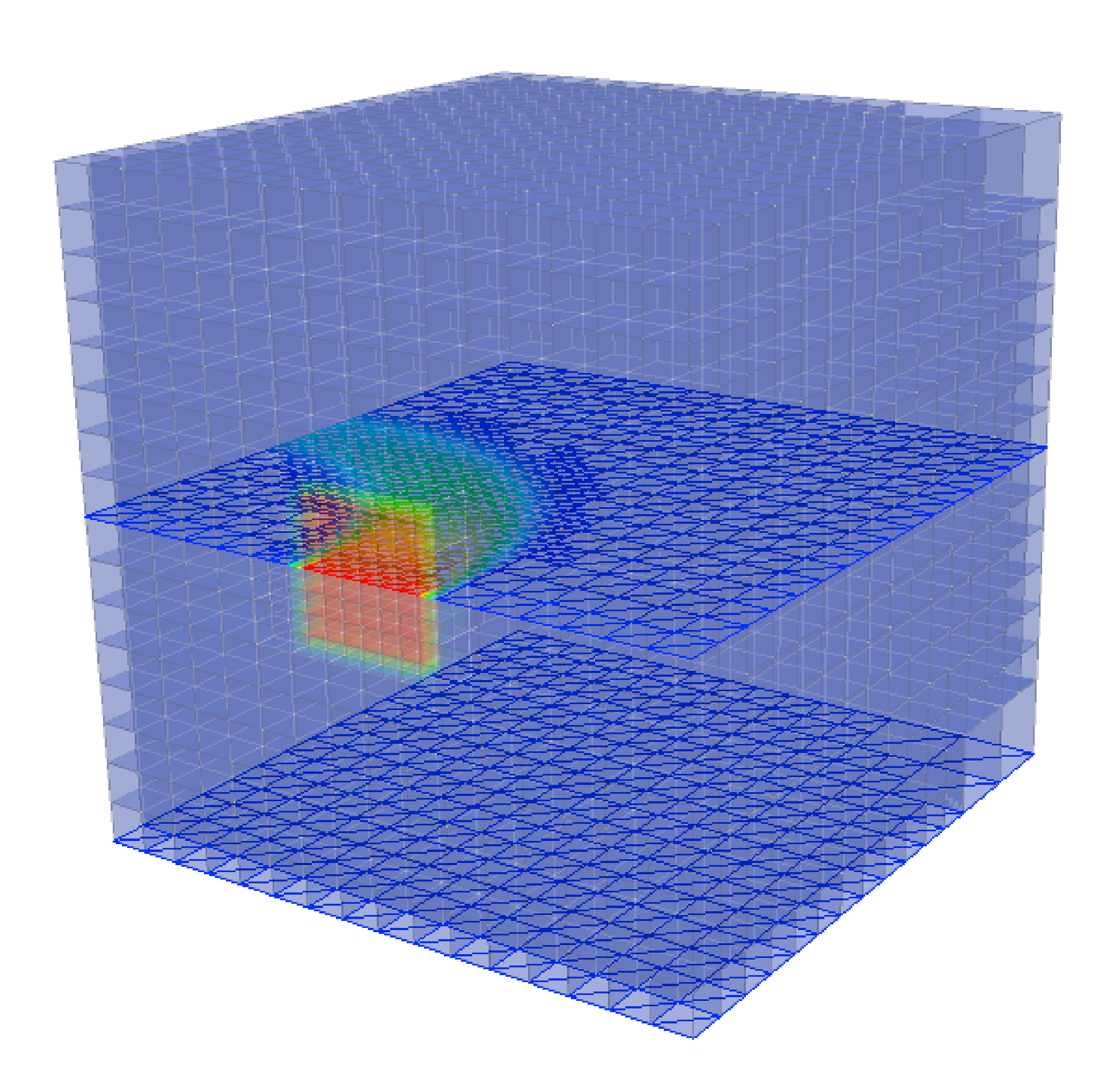}
\end{minipage}
\label{BruchhaeuserFig:61}
}
\hfill
\subfloat[]{
  \centering
\begin{minipage}{.45\linewidth}
\centering
\includegraphics[width=5.5cm]{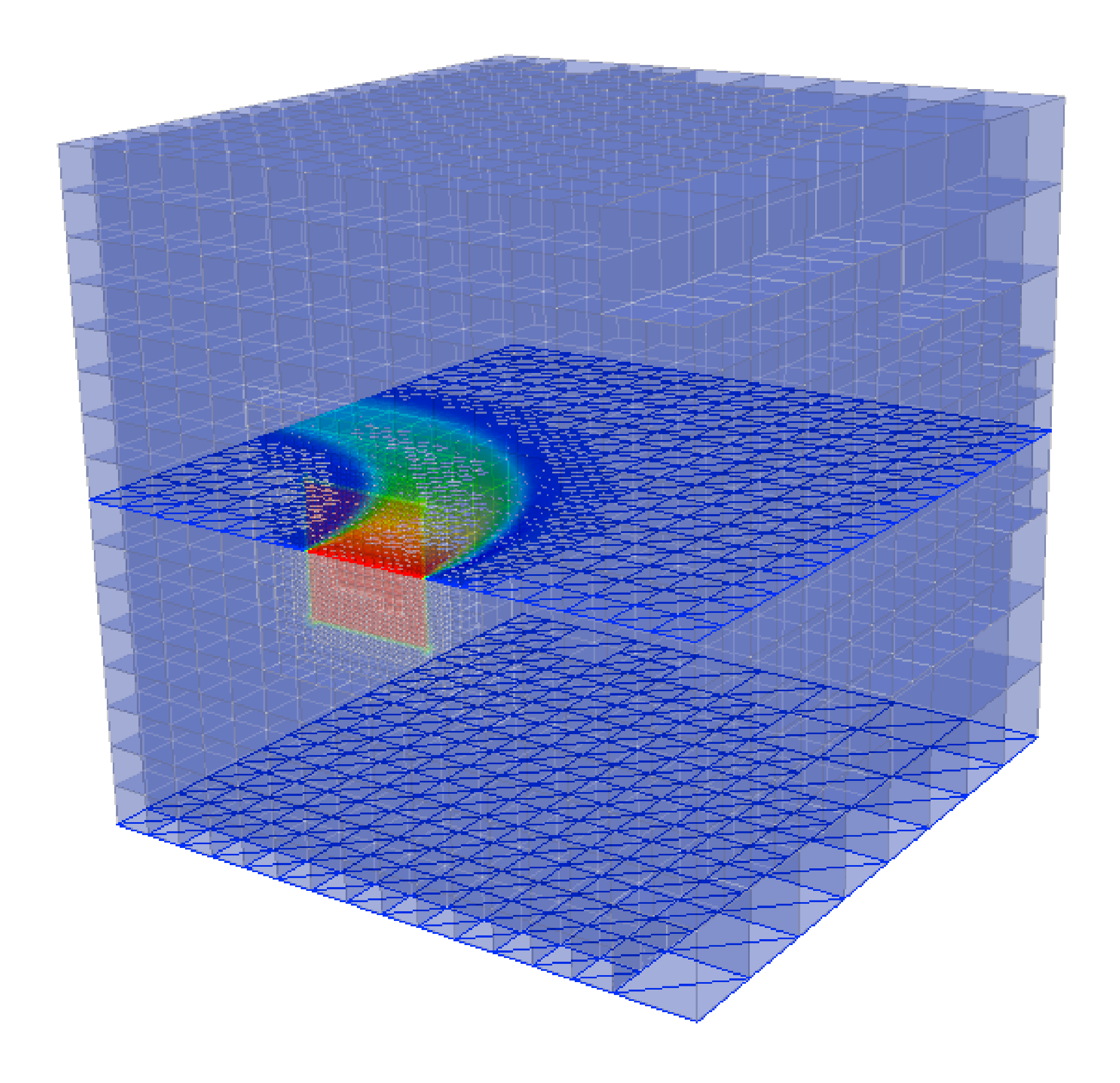}
\end{minipage}
\label{BruchhaeuserFig:63}
}

\subfloat[]{
  \centering
\begin{minipage}{.45\linewidth}
\centering
\includegraphics[width=5.5cm]{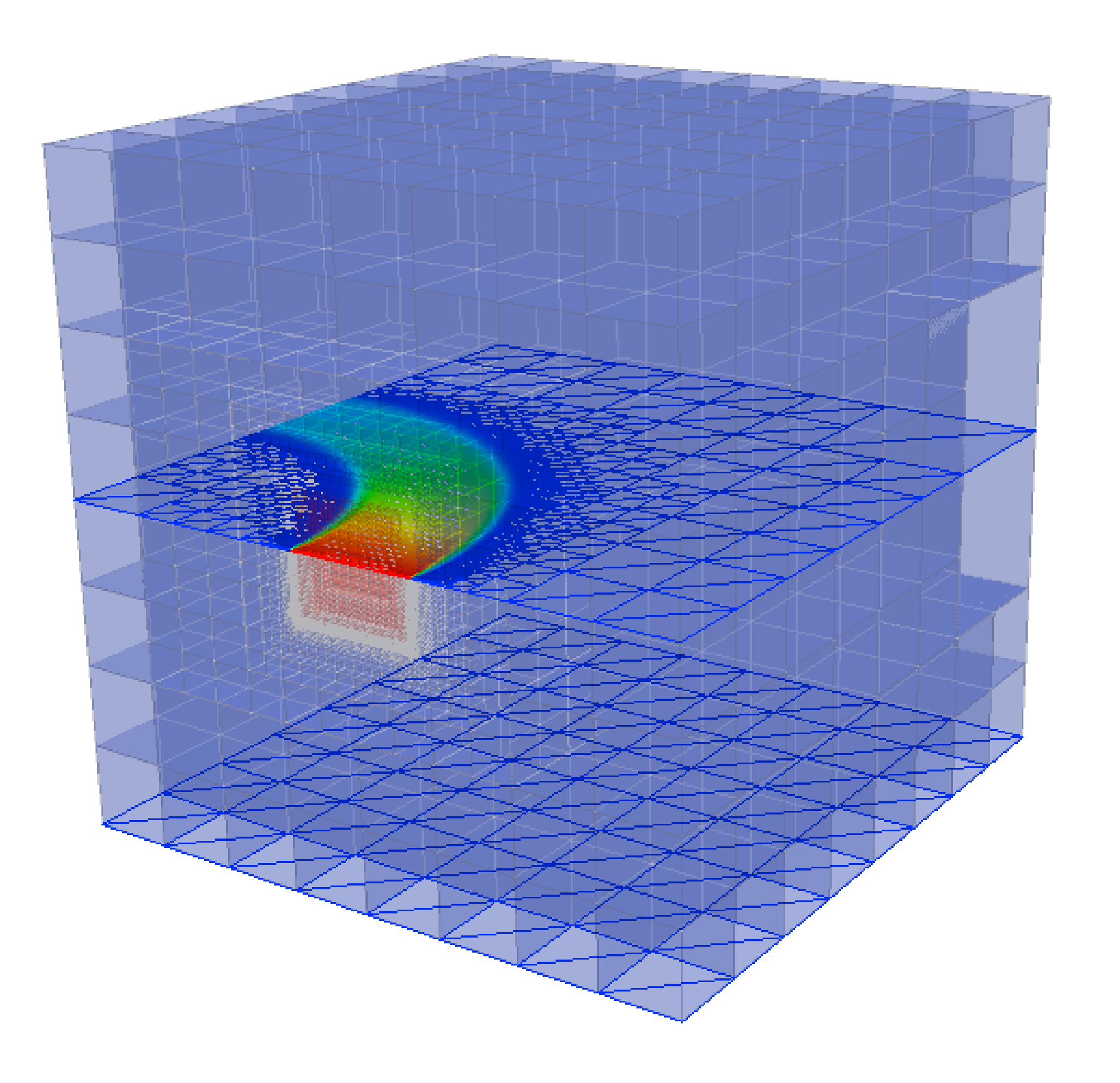}
\end{minipage}
\label{BruchhaeuserFig:66}
}
\hfill
\subfloat[]{
  \centering
\begin{minipage}{.45\linewidth}
\centering
\includegraphics[width=5.5cm]{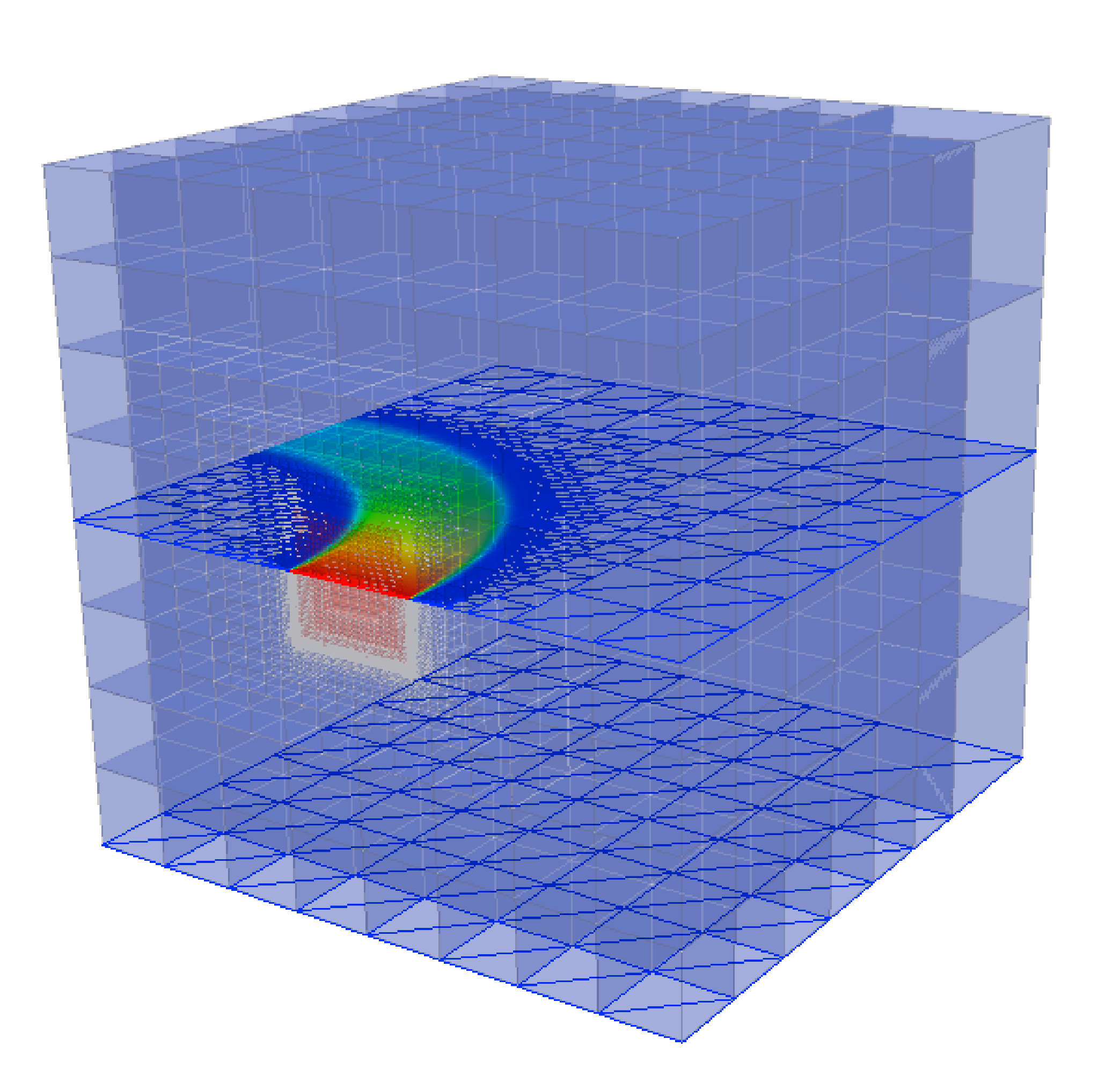}
\end{minipage}
\label{BruchhaeuserFig:67}
}
\caption{Adaptive grids after first (\ref{BruchhaeuserFig:61}), 
third (\ref{BruchhaeuserFig:63}), fifth (\ref{BruchhaeuserFig:66}) and 
seventh (\ref{BruchhaeuserFig:67}) iteration step of the DWR approach with coarsening and 
refinement with target functional $J_1$ for Example \ref{BruchhaeuserExample3}.}
\label{BruchhaeuserFig:VariableConField}
\end{figure}

\section{Summary}
\label{BruchhaeuserSec:summary}
In this work we developed an adaptive approach for stabilized finite element
approximations of stationary convection-dominated problems. It is based on the 
Dual Weighted Residual method for goal-oriented a posteriori error control. 
A \emph{first dualize and then stabilize} philosophy was applied for combining 
the mesh adaptation process in the course of the DWR approach with the 
stabilization of the finite element techniques. In contrast to other works of 
the literature we used a higher order 
approximation of the dual problem instead of a higher order interpolation of 
a lower order approximation of the dual solution. Thereby we aim to eliminate 
sources of inaccuracies in regions with layers and close to sharp fronts. 
In numerical experiments we could prove that spurious oscillations that 
typically arise in numerical approximations of convection-dominated problems 
could be reduced significantly. Robust effectivity indices very close to one were 
obtained for the specified test target quantities. We demonstrated the efficiency of the 
approach also for three space dimensions. The presented approach offers 
large potential for combining goal-oriented error control and selfadaptivity 
with stabilized finite element methods in the approximation of 
convection-dominated transport. The application of the approach to nonlinear, 
nonstationary and more sophisticated problems of multiphysics is our ongoing work. 
Moreover, the efficient computation of the higher order approximation to the dual problem 
offers potential for optimization. 
This will also be our work for the future. 

\input{ArXivBruchhaeuser_referenc}

\end{document}

%% file: ArXivBruchhaeuser_referenc.tex
%
%
%